\documentclass[12pt]{amsart}

\usepackage{amsmath}
\usepackage{amsthm}
\usepackage{amssymb}
\usepackage{amsfonts}
\usepackage[shortlabels]{enumitem}
\usepackage[dvipsnames]{xcolor}
\usepackage[normalem]{ulem}
\usepackage{marginnote}
\usepackage{ esint }
\usepackage{hyperref}

\makeatletter
\@namedef{subjclassname@2020}{%
 \textup{2020} Mathematics Subject Classification}
\makeatother

\theoremstyle{plain}
\newtheorem{theorem}{Theorem}[section]
\newtheorem{thm}{Theorem}

\newtheorem{lemma}[theorem]{Lemma}
\newtheorem{corollary}[theorem]{Corollary}
\newtheorem{proposition}[theorem]{Proposition}

\newtheorem*{claim*}{Claim}

\theoremstyle{definition}
\newtheorem*{definition*}{Definition}
\newtheorem{definition}[theorem]{Definition}
\newtheorem{example}[theorem]{Example}
\newtheorem{question}{Question}

\theoremstyle{remark}
\newtheorem{remark}[theorem]{Remark}

\numberwithin{equation}{section}

\def \D {\mathcal{D}}
\def \R {\mathbb{R}}
\def \A {\mathcal{A}}
\def \E {\mathbb{E}}
\def \H {\mathcal{H}}
\def \N {\mathbb{N}}

\def \X {\mathbb{X}}
\def \Y {\mathbb{Y}}
\def \L {\mathcal{L}}
\def \F {\mathcal{F}}
\def \B {\mathbb{B}}
\def \Z {\mathbb{Z}}
\def \eps {\varepsilon}

\DeclareMathOperator{\diam}{diam}

\DeclareMathOperator{\Lip}{Lip}
\DeclareMathOperator{\lip}{lip}

\DeclareMathOperator{\dist}{dist}
\DeclareMathOperator{\gen}{gen}

\DeclareMathOperator{\intr}{int}
\DeclareMathOperator{\aff}{aff}

\DeclareMathOperator{\At}{\mathcal{A}t}
\DeclareMathOperator{\Diff}{Diff}

\newcommand{\res}{\big|}

\begin{document}


\baselineskip=17pt


\title{Lipschitz functions on quasiconformal trees}

\author[D. Freeman]{David Freeman}
\address[D. Freeman]{University of Cincinnati Blue Ash College, Blue Ash, OH 45236, USA}
\email{freemadd@ucmail.uc.edu}

\author[C. Gartland]{Chris Gartland}
\address[C. Gartland]{Texas A\&M University, College Station, TX 77843, USA}
\email{cgartland@math.tamu.edu}

\date{}

\begin{abstract}
We first identify (up to linear isomorphism) the Lipschitz free spaces of quasiarcs. By decomposing quasiconformal trees into quasiarcs as done in an article of David, Eriksson-Bique, and Vellis, we then identify the Lipschitz free spaces of quasiconformal trees and prove that quasiconformal trees have Lipschitz dimension 1. Generalizing the aforementioned decomposition, we define a \emph{geometric tree-like decomposition} of a metric space. Our results pertaining to quasiconformal trees are in fact special cases of results about metric spaces admitting a geometric tree-like decomposition. Furthermore, the methods employed in our study of Lipschitz free spaces yield a decomposition of any (weak) quasiarc into rectifiable and purely unrectifiable subsets, which may be of independent interest.
\end{abstract}



\maketitle

\setcounter{tocdepth}{1}
\tableofcontents

\section{Introduction}
A \textit{Jordan arc} is a metric space homeomorphic to the unit interval $[0,1]$. Via \cite[Theorem 4.9]{TV80}, a \textit{quasiarc} can be defined as a Jordan arc that is both \textit{$D$-doubling} and \textit{$B$-bounded turning} (for some constants $D,B\geq1$). Here we say that a metric space $X$ is $D$-doubling if each metric ball in $X$ of radius $r>0$ can be covered by at most $D$ metric balls of radius $r/2$. The space $X$ is $B$-bounded turning if every pair of points $u,v\in X$ is contained in a compact and connected set $E$ such that $\diam(E)\leq B\,d(u,v)$.

Quasiarcs can be generalized by \textit{quasiconformal trees}. Let $T$ denote a \textit{metric tree}, meaning a compact, connected, locally connected metric space such that each pair of distinct points in $T$ forms the endpoints of a unique Jordan arc. We say that $T$ is a quasiconformal tree provided that $T$ is both doubling and bounded turning\footnote{Some sources, such as \cite{DEV21}, omit the requirement of local connectedness in the definition of a metric tree. However, if $T$ is a metric tree in this more general sense \emph{and $T$ is also bounded turning}, then $T$ is necessarily locally connected. Thus, the definition of quasiconformal trees from \cite{DEV21} coincides with ours.} that the Quasiconformal trees have been studied in papers such as \cite{K17}, \cite{BM20a}, \cite{BM20b}, \cite{DV20}, and \cite{DEV21}. 

For brevity, we shall refer to quasiarcs and quasiconformal trees as QC arcs and QC trees, respectively. Given a QC tree $T$, we say that $x\in T$ is a \textit{leaf} provided that $T\setminus\{x\}$ is connected. In particular, a QC arc is a QC tree possessing exactly two leaves.

In this paper, we study the Lipschitz geometry of QC arcs and trees. We do this in two ways. On one hand, we study the Banach space of real-valued Lipschitz functions on a given QC arc or tree, as well as its canonical Banach space predual. On the other, we focus on a special class of Lipschitz mappings known as \textit{Lipschitz light} mappings (as in \cite{CK13} and \cite{David21}), and show that QC trees admit Lipschitz light mappings into the real line $\mathbb{R}$. In other words, we prove that the \textit{Lipschitz dimension} of any QC tree is equal to 1. We expand on these concepts in more detail as follows.

\subsection{Lipschitz Free Spaces of QC Trees}
Given a metric space $(X,d)$ with fixed basepoint $x_0\in X$, we denote by $\Lip_0(X)$ the space of all Lipschitz functions $f:X\to\mathbb{R}$ such that $f(x_0)=0$. Here a function $f:X\to \mathbb{R}$ is said to be \textit{$L$-Lipschitz}, for some $L\geq 1$, provided that, for all $u,v\in X$, we have $|f(u)-f(v)|\leq L\,d(u,v)$. The space $\Lip_0(X)$ is a Banach space when equipped with the norm 
\[\|f\|_{\Lip_0(X)}:=\sup_{u\not=v}\frac{|f(u)-f(v)|}{d(u,v)}.\]
The \textit{Lipschitz free space} of $X$, denoted $\mathcal{F}(X)$, is the canonical Banach space predual of $\Lip_0(X)$ (and the unique predual when $X$ is bounded). See \cite{GK03} or \cite[Chapter 3]{Weaver18} for further background, and note that Lipschitz free spaces are also referred to as \textit{Arens-Eells spaces}.

In the study of free spaces, a major topic is embeddability into $L^1$-spaces\footnote{By an \emph{$L^1$-space}, we mean a Banach space of the form $L^1(\mu)$ for some measure $\mu$.}. It is known that $\F(X)$ linearly isometrically embeds into an $L^1$-space if and only if $X$ isometrically embeds into a geodesic tree (and equals an $L^1$-space if $X$ equals a geodesic tree) \cite{Godard10}, and $\F(X)$ linearly isometrically embeds into $\ell^1(\Lambda)$ for some indexing set $\Lambda$ if and only if $X$ is a subset of a geodesic tree such that both $X$ and the closure of the branch points of $X$ have length measure zero\footnote{A subset $X$ of a geodesic tree $T$ has \emph{length measure zero} if the Lebesgue measure of $[x,y] \cap X$ equals 0 for every subarc $[x,y] \subset T$.} \cite{APP}. However, the problem of isomorphic embeddability has thus far proved to be more difficult; as far as we are aware, there does not even exist a conjectural characterization of metric spaces whose free space isomorphically embeds into an $L^1$-space. In addition to subsets of geodesic trees, it is known that whenever $(X,d)$ is doubling and $\alpha \in (0,1)$, the free space $\F(X,d^\alpha)$ isomorphically embeds into an $L^1$-space (see \cite[Theorem~8.49]{Weaver18} for a proof that $\F(X,d^\alpha)$ is isomorphic to $\ell^1$), in spite of the fact that these spaces are quite different geometrically from subsets of trees. On the other hand, there are doubling spaces whose free spaces do not isomorphically embed into $L^1$. For example, $\F([0,1]^2)$ and $\F(\Z^2)$ fail to embed into $L^1$ by Naor-Schechtman's adaptation of Kislyakov's argument \cite{NaorSchechtman}.

The canonical examples of QC arcs are the Euclidean interval $[0,1]$ and its snowflakes $([0,1],|\cdot|^\alpha)$. As mentioned above, $\F([0,1])$ is linearly isometric to an $L^1$-space and $\F([0,1],|\cdot|^\alpha)$ is isomorphic to $\ell^1$ if $\alpha \in (0,1)$. Our first main theorem simultaneously generalizes these results (see the comment before Corollary \ref{cor:isoL1} and Subsection \ref{SS:RP1U} for the definition of purely 1-unrectifiable).

\begin{thm}\label{T:main_arcs}
For every QC arc $\gamma$, the Lipschitz free space $\F(\gamma)$ is linearly isomorphic to $L^1(Z)$ for some measure space $Z$. Moreover, $Z$ is purely atomic if and only if $\gamma$ is purely 1-unrectifiable.
\end{thm}

\noindent Each of the isomorphisms $\F([0,1]) \cong L^1([0,1])$ and $\F([0,1],|\cdot|^\alpha) \cong \ell^1(\mathcal{V})$ can be obtained by predualizing a weak*-isomorphism $\Lip_0([0,1]) \to L^\infty([0,1])$ and $\Lip_0([0,1],|\cdot|^\alpha) \to \ell^\infty(\mathcal{V})$. In the first case, the map is simply the usual continuous derivative afforded by Lebesgue's theorem, and in the second case, the map uses a sequence of discrete derivatives. The proof of our Theorem \ref{T:main_arcs} follows along the same lines; we construct a linear map on $\Lip_0(\gamma)$ by using derivatives. However, because $\gamma$ lacks the spatial- and scale-homogeneity enjoyed by the arcs $[0,1]$ and $([0,1],|\cdot|^\alpha)$, it is not clear exactly what kind of derivatives should be taken or how to use them to form a weak*-isomorphism onto an $L^\infty$-space. A new insight we provide is to first use derivatives to map $\Lip_0(\gamma)$ isomorphically onto a subspace of an $L^\infty$-space whose underlying measure space is adapted to the non-homogeneous geometry of $\gamma$, and then map this subspace isomorphically onto another $L^\infty$-space with a more abstract map. The intermediate Banach space we use is the space of \emph{$L^\infty$-bounded martingale difference sequences adapted to a filtration $(\A_n)_{n\geq0}$} of $\sigma$-algebras on $[0,1]$. In a sense, this filtration measures the deviation of $\gamma$ from the Euclidean interval $[0,1]$. The key tool used to construct $(\A_n)_{n\geq0}$ is a combinatorial description of QC arcs provided by work of Herron and Meyer \cite{HM12} (Theorem \ref{T:HM}). The bulk of Section \ref{S:arc_freespace} is dedicated to proving Theorem \ref{thm:Diso}, which establishes the weak*-isomorphism between $\Lip_0(\gamma)$ and the space of $L^\infty$-bounded martingale difference sequences adapted to $(\A_n)_{n\geq0}$. From this theorem we deduce Corollary \ref{cor:isoL1} (a restatement of Theorem \ref{T:main_arcs}) with relative ease.

With intermediate results used to obtain Theorem \ref{T:main_arcs}, we are able to produce a rectifiable/purely unrectifiable decomposition (see Subsection \ref{SS:RP1U} for the definitions) for bounded turning Jordan arcs a l\`a \cite[15.6 Theorem]{Mattila}.

\begin{thm}\label{T:rp1udecomp}
For every bounded turning Jordan arc $\gamma$, there exist $R,U \subset \gamma$ such that $R$ is countably 1-rectifiable, $U$ is purely 1-unrectifiable, and $\gamma = R \cup U$.
\end{thm}

\noindent We note that such a decomposition is typically available only for $\H^1$-$\sigma$-finite metric spaces. Theorem \ref{T:rp1udecomp} is restated and proven as Theorem \ref{thm:rp1udecomp} in Subsection \ref{SS:RP1U}.

We build on Theorem \ref{T:main_arcs} and study Lipschitz functions on QC trees. We do this by way of a decomposition provided by \cite{DEV21}. Given a QC tree $T$, we refer to such a decomposition as a \textit{DEBV decomposition of $T$} (see Definition \ref{D:DEBV}). In particular, DEBV decompositions allow us to break down a QC tree into constituent QC arcs. Denoting these QC arcs by $\{\gamma_i\}_{i\in I}$, our second main result can be presented as follows.

\begin{thm}\label{T:main_trees}
For every QC tree $T$, the Lipschitz free space $\F(T)$ is linearly isomorphic to the $\ell^1$-sum $\bigoplus_{i\in I}^1\F(\gamma_i)$. Consequently, $\F(T)$ is isomorphic to $L^1(Z)$ for some measure space $Z$, where $Z$ is purely atomic if and only if $T$ is purely 1-unrectifiable.
\end{thm}

The moral of Theorem \ref{T:main_trees} is that Lipschitz functions on a QC tree can be decomposed into Lipschitz functions on certain QC arcs comprising the tree (which can then be studied via Theorem \ref{T:main_arcs}). The core of the proof heavily relies on geometric properties of a DEBV decomposition. Theorem \ref{T:main_trees} follows from Corollary \ref{C:tree_arcs} in Section \ref{S:tree_apps}.

We point out that examples of QC arcs not bi-Lipschitz homeomorphic to $[0,1]$ or any snowflake space $(X,d^\alpha)$ (for $\alpha \in (0,1)$) are provided by arcs satisfying certain types of \textit{generalized chord-arc conditions} as investigated in \cite{HM99}. For specific examples of such arcs, see \cite[Proposition 1.10]{LD13} or \cite[Section 8]{Freeman10}. In particular, we emphasize that Theorems \ref{T:main_arcs} and \ref{T:main_trees} are not contained in results such as \cite[Theorem 3.2]{Godard10} or \cite[Corollary 5.7]{AACD21}.

\subsection{Lipschitz Dimension of QC Trees}
Next, we focus our attention on a particular subclass of Lipschitz functions known as Lipschitz light mappings. These mappings were  introduced by Cheeger and Kleiner in \cite{CK13} and further developed by David in \cite{David21}. 

Lipschitz light maps are defined as follows. Given $\delta>0$, we say that a sequence $(u_i)_{i\in I}$ is a \textit{$\delta$-chain} provided that, for each $i<\max(I)$, we have $d(x_i,x_{i+1})\leq \delta$. A subset $U$ of a metric space $X$ is \textit{$\delta$-connected} if every pair of points in $U$ is contained in a $\delta$-chain in $U$. A \textit{$\delta$-component} of $X$ is a maximal $\delta$-connected subset of $X$. 

\begin{definition}\label{D:LL}
A map $f:X\to Y$ between metric spaces is \textit{Lipschitz light} if there exist constants $L,Q>0$ such that \begin{enumerate}
    \item{$f$ is $L$-Lipschitz, and}
    \item{for every $r>0$ and $E\subset Y$ such that $\diam(E)\leq r$, the $r$-components of $f^{-1}(E)$ have diameter at most $Qr$.}
\end{enumerate}
In this case we say that $f$ is $L$-Lipschitz and $Q$-light. 
\end{definition}

Given a collection of maps $\{f_n\}_{n\in N}$, we say that the maps are \textit{uniformly} Lipschitz light if there exist $L,Q>0$ such that, for every $n\in N$, the map $f_n$ is $L$-Lipschitz and $Q$-light.

\begin{remark}\label{R:LipLight}
In \cite[Section 1.4]{David21}, David points out that the above definition of a Lipschitz light map is equivalent to the following for maps into Euclidean space: There exist $L,Q>0$ such that $f$ is $L$-Lipschitz, and, for every bounded subset $E\subset \mathbb{R}^d$, the $\diam(E)$-components of $f^{-1}(E)$ have diameter at most $Q\cdot\diam(E)$.
\end{remark}

Lipschitz light mappings are used to define the \textit{Lipschitz dimension} of a metric space. Specifically, a metric space $X$ has Lipschitz dimension at most $n$ provided there exists a Lipschitz light map $f:X\to \mathbb{R}^n$. A few reasons that Lipschitz dimension is of theoretical significance are provided by the embedding results for spaces of Lipschitz dimension 1 contained in \cite{CK13} and certain non-embedding results about spaces of infinite Lipschitz dimension contained in \cite{David21}. 

In \cite{Freeman20}, the first author proves that the Lipschitz dimension of any bounded turning Jordan arc is equal to 1. We build on this work via our third main result. 

\begin{thm}\label{T:main_dim}
The Lipschitz dimension of any QC tree is equal to $1$.
\end{thm}

\noindent Theorem \ref{T:main_dim} follows from Theorem \ref{T:tree_dim1} in Section \ref{S:tree_apps}.

\subsection{Spaces Admitting a Geometric Tree-Like Decomposition}
The above results pertaining to QC trees are in fact special cases of results we obtain for metric spaces that admit a geometric tree-like decomposition. We define such a decomposition as follows.

\begin{definition}\label{D:treelike}
We say that $\{X_n\}_{n\in N}$ is a \textit{tree-like decomposition} of a set $X$ (with indexing set $N = \{0,1,\dots \max(N)\} \subset \N$ or $N = \N$) if $X=\bigcup_{n\in N}X_n$ and, for each $n \in N \setminus \{0\}$, there exists a unique point $p_n\in X_n\cap\bigcup_{m<n}X_m$. The points $\{p_n\}_{n\in N\setminus\{0\}}$ are called \textit{branch points}. 
\end{definition}

This definition is inspired by the construction described in \cite[Lemma 3.12]{Weaver18}. Note that \textit{any} space $X$ admits the (trivial) tree-like decomposition $\{X\}$. Given a tree-like decomposition, we can then define the following.

\begin{definition}\label{D:decomp_path}
Suppose $X$ is a metric space with tree-like decomposition $\{X_n\}_{n\in N}$. Given points $x,y\in X$, a sequence $(z_i)_{i\in I}$ (with indexing set $I = \{0,1,\dots \max(I)\} \subset \N$) is a \textit{decomposition path} from $x$ to $y$ provided that 
\begin{enumerate}
    \item{$z_0=x$ and $z_{\max(I)}=y$,}
    \item{$\{z_i\,|\,1\leq i\leq \max(I)-1\}$ consists solely of branch points, and}
    \item{for each $1\leq i\leq \max(I)$, there exists $n_i\in N$ such that $\{z_{i-1},z_{i}\}\subset X_{n_i}$.}
\end{enumerate}
Furthermore, a decomposition path $(z_i)_{i\in I}$ is \textit{minimal} if there does not exist any proper subset $I'\subsetneq I$ such that $(z_i)_{i\in I'}$ forms a decomposition path from $z_0$ to $z_{\max(I)}$. Obviously, any decomposition path from $x$ to $y$ contains a minimal decomposition path from $x$ to $y$.
\end{definition}

\begin{remark}\label{R:minimal}
We note that $(z_i)_{i\in I}$ is minimal if and only if $z_j\not\in X_{n_i}$ for all $i<j$ in $I$.
\end{remark}

Via the above definitions we are now able to state the following. 

\begin{definition}\label{D:decomp}
Given a metric space $X$ and a constant $C\geq 1$, we say that $\{X_n\}_{n\in N}$ is a \emph{$C$-geometric} tree-like decomposition provided that 
\begin{enumerate}
    \item{$\{X_n\}_{n\in N}$ is a tree-like decomposition,}
    \item{for any minimal decomposition path $(z_i)_{i\in I}$, we have $$ d(z_0,z_{\max(I)})\leq C\max_{i< \max(I)}d(z_{i},z_{i+1}),$$ and,}
    \item{given $x,y\in X$, there exists a \textit{short} decomposition path $(z_i)_{i\in I}$ from $x$ to $y$. That is, a decomposition path satisfying 
    $$C^{-1}\sum_{i< \max(I)}     d(z_{i},z_{i+1})\leq d(z_0,z_{\max(I)}).$$}\label{E:short}
\end{enumerate}
\end{definition}


\begin{remark} \label{rmk:finitetreedecomp}
Note that, if $C < \infty$ and $\{X_n\}_{n \in \N}$ is a sequence of subsets of $X$ such that $\{X_n\}_{n \leq m}$ is a $C$-geometric tree-like decomposition of $\bigcup_{n \leq m} X_n$ for infinitely many $m \in \N$, then $\{X_n\}_{n \in \N}$ is a $C$-geometric tree-like decomposition of $X$.
\end{remark}

Under the assumption that $X$ admits a geometric tree-like decomposition, we obtain the following results. The first is reminiscent of \cite[Lemma 3.12]{Weaver18} and is implied by Theorem \ref{T:general} in Section \ref{S:Lip_and_QA}. 

\begin{thm}\label{T:geometric_Lip}
Given a metric space $X$ admitting a geometric tree-like decomposition $\{X_n\}_{n\in N}$, the space $\Lip_0(X)$ is weak*-isomorphic to the $\ell^\infty$-sum $\bigoplus_{n\in N}^\infty \Lip_0(X_n)$.
\end{thm}

Building on Theorem \ref{T:geometric_Lip}, we also obtain the following result, which is implied by Theorem \ref{T:general_LL} in Section \ref{S:Lip_and_QA}.

\begin{thm}\label{T:geometric_dim}
Given $d\geq1$ and a metric space $X$ admitting a geometric tree-like decomposition $\{X_n\}_{n\in N}$, if there exist uniformly Lipschitz light maps $f_n:X_n\to \mathbb{R}^d$, then $\dim_L(X)\leq d$. 
\end{thm}

\begin{remark}
Via \cite[Theorem 1.7]{CK13}, it follows that a metric space satisfying the hypotheses of Theorem \ref{T:geometric_dim} with $d=1$ admits a bi-Lipschitz embedding into $L^1(Z)$ for some measure space $Z$. Regarding bi-Lipschitz embeddings into Banach spaces, we also point the reader to Corollary \ref{C:Banachembed} in Subsection \ref{s:banach}.
\end{remark}

\section{Preliminaries}

Here we present relevant notation and definitions along with a few requisite lemmas.

\subsection{Dual Spaces and the Weak*-Topology}
We start with a review of general Banach space theory surrounding dual spaces and the weak*-topology. We will use these common notions frequently in Section \ref{S:arc_freespace}, but also in a few other places.

Many Banach spaces $\X$ are canonically isometrically identified with the dual of another Banach space $\X_*$, which we call the canonical \emph{predual} of $\X$. The only examples of such spaces we will consider in this paper are $L^\infty(\mu)_* = L^1(\mu)$ and $\Lip_0(X)_* = \F(X)$ for $\mu$ a $\sigma$-finite measure and $X$ a metric space, and also new spaces with canonical preduals built from these ones, as explained in the next paragraph. We say that a bounded linear map $T: \X \to \Y$ between spaces with canonical preduals is weak*-weak*-continous if it is continuous with respect to the weak*-topologies induced by $\X_*$ and $\Y_*$. By the Krein-Smulian theorem, $T$ is weak*-weak*-continuous if and only if its restriction to the unit ball of $\X$ is weak*-weak*-continuous.

We now present two constructions of new spaces with canonical preduals from old ones. Given a sequence of Banach spaces $\X_n$ and $p \in [1,\infty]$, we denote its $\ell^p$-sum space by $\bigoplus_n^p \X_n$. Given a closed subspace $\Y \subset \X$ of a space with a canonical predual, we write $\Y_\perp := \{x_* \in \X_*: \forall y \in \Y, \: y(x_*) = 0\}$ for the pre-annihilator. If $\X_n$ is a sequence of spaces with canonical preduals $(\X_n)_*$, then $\left(\bigoplus_n^\infty \X_n\right)_* = \bigoplus_n^1 (\X_n)_*$, in the sense that $\bigoplus_n^\infty \X_n$ is isometrically identified with the dual space of $\bigoplus_n^1 (\X_n)_*$ and a bounded linear map $T = (T_n)_n: \Y \to \bigoplus_n^\infty \X_n$, where $\Y$ has a canonical predual, is weak*-weak*-continuous if and only if $T_n$ is weak*-weak*-continuous for all $n$. If $\X$ has a canonical predual and $\Y \subset \X$ is weak*-closed, then $(\Y)_* = (\X_*)/\Y_\perp$, in the sense that $\Y$ is isometrically identified with the dual space of $(\X_*)/\Y_\perp$ and the inclusion $\Y \hookrightarrow \X$ is weak*-weak*-continuous.

A bounded linear map $T: \X \to \Y$ between spaces with canonical preduals is the adjoint of a map $T_*: \Y_* \to \X_*$ if and only if it is weak*-weak*-continuous. In this case, $T_*$ is unique and we call it the \emph{predual} of $T$. If $T$ is weak*-weak* continuous, it holds that $T$ is a $C$-isomorphism if and only if $T_*$ is a $C$-isomorphism, where we call a linear operator $S$ a \emph{$C$-isomorphism} if it is bijective and $\|S\|\|S^{-1}\| \leq C$. We call weak*-weak*-continuous isomorphisms \emph{weak*-isomorphisms.}

Let $X$ be a separable metric space. Then $\F(X)$ is separable as well, and therefore the unit ball of $\Lip_0(X) = \F(X)^*$ equipped with the weak*-topology is metrizable. This implies that continuity of maps on the unit ball of $\Lip_0(X)$ can be checked with weak*-convergent sequences. On the unit ball of $\Lip_0(X)$, the weak*-topology coincides with the topology of pointwise convergence. Putting this all together, and using facts from the preceding paragraphs, we get that, for $T: \Lip_0(X) \to \Y$ a bounded linear map, where $\Y$ has a canonical predual, the following are equivalent:
\begin{itemize}
    \item $T$ is weak*-weak*-continuous.
    \item There exists a predual map $T_*: \Y_* \to \F(X)$.
    \item $T(f_n)$ weak*-converges to $T(f)$ whenever $f_n$ is a sequence in $\Lip_0(X)$ pointwise-converging to some $f \in \Lip_0(X)$ such that $\sup_n \|f_n\|_{\Lip_0(X)}$ $\leq 1$.
\end{itemize}

\subsection{Sets and Sequences}
We write $\{x_i\}_{i\in I}$ to denote a set of points indexed by elements of $I$, where $I$ denotes either a finite set of non-negative integers $\{0,1,\dots,\max(I)\}$ or the infinite set $\{n\in \mathbb{Z}\,|\,n\geq0\}$. On occasion we may set $I=\{1,\dots,\max(I)\}$, but, unless stated otherwise, the minimal element of any index set is $0$. 

Given a set $\{x_i\}_{i\in I}\subset X$, we define the sequence $(x_i)_{i\in I}$ to be $(x_0,\dots,$ $x_{\max(I)})$. We define the concatenation of two sequences as 
\[(x_i)_{i\in I}*(y_j)_{j\in J}:=\begin{cases}
    (x_0,\dots,x_{\max(I)},y_0,\dots,y_{\max(J)})    & x_{\max(I)}\not=y_0\\
    (x_0,\dots,x_{\max(I)},y_1,\dots,y_{\max(J)})    & x_{\max(I)}=y_0.
    \end{cases}\]

At times it will be convenient to concatenate decomposition paths. This is accomplished as follows. Suppose $x,y,y',z$ are points in $X$ such that $\{y,y'\}$ is contained in a single element $X_m\in\{X_n\}_{n\in N}$. Given decomposition paths $(z_i)_{i\in I}$ from $x$ to $y$ and $(w_j)_{j\in J}$ from $y'$ to $z$, the concatenation 
\[(z_i)_{i\in I}\hat{*}(w_j)_{j\in J}=(u_k)_{k\in K}\]
is defined by removing at most two points, namely $y$ and/or $y'$, from $(z_i)_{i\in I}*(w_j)_{j\in J}$ such that $(u_k)_{k\in K}$ forms a decomposition path from $x$ to $z$. That is, $(z_i)_{i\in I}\hat{*}(w_j)_{j\in J}$ is the maximal subsequence of $(z_i)_{i\in I}*(w_j)_{j\in J}$ that forms a decomposition path from $x$ to $z$.

\subsection{Tree-Like Decompositions}\label{SS:QA}
Suppose $(z_i)_{i\in I}$ is a decomposition path in $X$ such that $z_0=z_{\max(I)}$. Furthermore, suppose $z_i\not=z_j$ and $n_i\not= n_j$ for $1\leq i\not= j\leq \max(I)$ (except perhaps for $n_1 = n_{\max(I)}$). We call such a decomposition path a \textit{simple decomposition loop} with basepoint $z_0$. A simple decomposition loop is \textit{trivial} if $\max(I)=2$ and $X_{n_1}=X_{n_2}$. 

We provide the following lemma to justify our use of the terminology \textit{tree-like} decomposition, and for use in the proof of Lemma \ref{L:gen_chain}.

\begin{lemma}\label{L:unique}
Let $X$ denote a metric space with tree-like decomposition $\{X_n\}_{n\in N}$. Then $X$ contains no non-trivial simple decomposition loops.
\end{lemma}

\begin{proof}
Suppose $(z_i)_{i\in I}$ is a non-trivial simple decomposition loop in $X$. First assume $\max(I) \geq 3$. We claim that there exists $1 \leq i^* \leq \max(I)$ such that $i \mapsto n_i$ is strictly decreasing on $\{1, \dots i^*\}$ and strictly increasing on $\{i^*, \dots \max(I)\}$. Indeed, otherwise there must exist $2\leq i\leq \max(I)-1$ such that $n_{i-1},n_{i+1} < n_i$. This implies $z_{i-1},z_i \in X_{n_i} \cap \bigcup_{n<n_i} X_n$, from which we get $z_{i-1} = z_i$ by the definition of tree-like decomposition, which in turn contradicts the definition of simple decomposition loop. This proves the claim. Notice that this implies
\begin{equation}\label{E:ineq_alts}
    \text{either } n_{\max(I)} > n_{\max(I)-1} \text{ or } n_1 > n_2.
\end{equation}
If both inequalities in (\ref{E:ineq_alts}) hold, then $z_1 \in X_{n_1} \cap X_{n_2}$ and $z_{\max(I)-1} \in X_{n_{\max(I)-1}} \cap X_{n_{\max(I)}}$. This would then imply that
\begin{equation}\label{eq:branch}
   z_1 \in X_{n_1} \cap \bigcup_{n < n_1} X_n \text{ and } z_{\max(I)-1} \in X_{n_{\max(I)}} \cap \bigcup_{n < n_{\max(I)}} X_n.
\end{equation}

Now we split into two cases: $n_1 = n_{\max(I)}$ and $n_1 \neq n_{\max(I)}$. Assume the first case holds. Then both inequalities of (\ref{E:ineq_alts}) hold, and so \eqref{eq:branch} implies $z_1,z_{\max(I)-1} \in X_{n_1} \cap \bigcup_{n < n_1} X_n$, from which we get $z_{1} = z_{\max(I)-1}$ by the definition of tree-like decomposition. This contradicts the definition of simple decomposition loop (since $\max(I) \geq 3$). Now assume the second case $n_1 \neq n_{\max(I)}$ holds. Then either $n_1 > n_{\max(I)}$ or $n_1<n_{\max(I)}$. We assume $n_1>n_{\max(I)}$ (and note that the alternative can be treated analogously). This assumption implies that $n_1>n_2$ as in (\ref{E:ineq_alts}), and thus that $z_1\in X_{n_1}\cap \bigcup_{n < n_1} X_n$ as in (\ref{eq:branch}). Since $X_{n_{\max(I)}} \ni z_{\max(I)} = z_0 \in X_{n_1}$, we also get $z_{\max(I)} \in X_{n_1} \cap \bigcup_{n < n_1} X_n$. The definition of tree-like decomposition then implies $z_{\max(I)} = z_1$. This contradicts the definition of simple decomposition loop and completes the proof in the case $\max(I) \geq 3$.

Finally, assume $\max(I) = 2$. Then $z_2 = z_0$, and this easily implies $z_0,z_1,z_2 \in X_{n_1} \cap X_{n_2}$. Then the definition of tree-like decomposition implies $z_0=z_1=z_2$, which contradicts the definition of a simple decomposition loop.
\end{proof}

\subsection{The Herron-Meyer Catalogue}
For each $k \geq 0$, let $\D_k$ denote the set of \emph{dyadic intervals} in $[0,1]$ of level $k$, which we refer to as dyadic \textit{$k$-edges}. That is, $\D_k := \{[(j-1)2^{-k},j2^{-k}] \subset [0,1]: 1 \leq j \leq 2^k\}$. Set $\D := \bigcup_{k=0}^\infty \D_k$. We write $\gen([0,1]) := 0$, and, if $e\in \D_k \setminus \D_{k-1}$ for some $k\geq 1$, $\gen(e):=k$. 

For each $k\geq 0$, we write $\mathcal{V}_k$ to denote the collection of endpoints of $k$-edges in $\mathcal{D}_k$. That is, $\mathcal{V}_k=\{j2^{-k}\,:\,0\leq j\leq 2^k\}$. Set $\mathcal{V}=\bigcup_{k=0}^\infty \mathcal{V}_k$. Thus $\mathcal{V}_k$ is the collection of endpoints of dyadic $k$-edges, while $\mathcal{V}$ is the collection of endpoints of all dyadic edges.

We use the language of a dyadic tree to describe edges in $\D$. In particular, given any $e \in \D$, there are exactly two dyadic \textit{children} edges contained in $e$, and $e$ is contained in its unique dyadic \textit{parent} edge. Two children with the same parent are called \textit{siblings}. If $e\in\D$ is strictly contained in  $e'\in\D$, we say that $e$ is a \textit{descendent} of $e'$ and that $e'$ is an \emph{ancestor} of $e$. 

Following \cite{HM12}, we call a function $\Delta:\D\to(0,1]$ a \textit{dyadic diameter function} provided that $\Delta([0,1])=1$ and, for any $e\in\D$, either 
\[\Delta(e_0)=\Delta(e_1)=\frac{1}{2}\Delta(e) \quad \text{or} \quad \Delta(e_0)=\Delta(e_1)=\Delta(e).\] 
Here, and in what follows, $e_0$ and $e_1$ denote the two dyadic children of $e$, with $\max(e_0) = \min(e_1)$. We also require that 
\begin{equation}\label{E:smaller}
\lim_{n\to+\infty}\max\{\Delta(e)\,|\,e\in\mathcal{D}_n\}=0.
\end{equation}
Note that, differing from \cite{HM12}, we omit the parameter $\sigma$ from the definition of $\Delta$. This is because $\sigma=1$ for all dyadic diameter functions utilized in our context. 

Denote the collection of all such dyadic diameter functions as $\mathfrak{D}$. For every $\Delta\in\mathfrak{D}$, the function $d_\Delta$ on $[0,1]\times[0,1]$ is defined as 
\[d_\Delta(x,y):=\inf\sum_{k\in N}\Delta(e^k),\]
where the infimum is taken over all \textit{dyadic chains} $(e^k)_{k\in N}\subset\D$ (with indexing set $N = \{0,1, \dots \max(N)\}$) such that
\begin{enumerate}
    \item{$\{x,y\}\subset\bigcup_{k\in N}e^k$}, and
    \item{for all $k \in N \setminus \{\max(N)\}$, we have $e^{k}\cap e^{k+1}\not=\emptyset$.}
\end{enumerate}
By \cite[Lemma 3.1]{HM12}, the function $d_\Delta$ is a distance, and the metric space $([0,1],d_\Delta)$ is a $1$-bounded turning Jordan arc.
\begin{example}
If $\Delta \in \mathfrak{D}$ is defined by $\Delta(e) := 2^{-n}$ for every $e \in \D_n$, then $d_\Delta$ is simply the Euclidean metric on $[0,1]$.

If $\Delta(e) := 2^{-n/2}$ for all $e \in \D_n$ with $n$ even and $\Delta(e) := 2^{-(n-1)/2}$ for all $e \in \D_n$ with $n$ odd, then $d_\Delta$ is bi-Lipschitz equivalent to the snowflake $([0,1],|\cdot|^\alpha)$ for some $0<\alpha<1$.
\end{example}
We write $\mathcal{S}_1'$ to denote the collection of all such arcs, each given by some dyadic diameter function $\Delta$. That is,
\[\mathcal{S}_1':=\{([0,1],d_\Delta)\,|\,\Delta\in\mathfrak{D}\}.\]
With this notation in hand, we are now ready to state the following result of Herron and Meyer, which is fundamental to our proof of Theorem \ref{T:main_arcs}.

\begin{theorem}[\cite{HM12}]\label{T:HM}
If $\gamma$ is a $B$-bounded turning Jordan arc, then $\gamma$ is $L$-bi-Lipschitz equivalent to an arc in $\mathcal{S}_1'$. Here $L=8B\max\{\diam(\gamma),$ $\diam(\gamma)^{-1}\}$.
\end{theorem}

The following lemma is analogous to \cite[Lemma 8.41]{Weaver18}. It provides a ``local-to-global" estimate for the Lipschitz constant of a function on a QC arc. The doubling property is crucial in this estimate. Moreover, this is the only lemma which directly relies upon the doubling property.

\begin{lemma}\label{L:dyadic}
Suppose $\Delta$ is a dyadic diameter function such that $([0,1],d_\Delta)$ is doubling, $f:[0,1]\to\mathbb{R}$ is a continuous function, and $L\geq1$ is a constant such that
\[|f(x)-f(y)|\leq L\,d_\Delta(x,y)=L\,\Delta([x,y])\]
whenever $[x,y]\in\mathcal{D}$. Then $f$ is $L'L$-Lipschitz with respect to $d_\Delta$, where $L'$ depends only on the doubling constant of $([0,1],d_\Delta)$.
\end{lemma}

\begin{proof}
By continuity of $f$ and density of $\mathcal{V} \subset [0,1]$, it suffices to prove the Lipschitz condition for pairs $x,y \in \mathcal{V}$. Let $x,y\in\mathcal{V}$ be such that $x\not=y$. Let $(e^j)_{j\in N}=([x_j,y_j])_{j\in N}$ denote a dyadic chain such that 
\begin{enumerate}
    \item{$\{x,y\}\subset \bigcup_{j\in N}e^j$,}
    \item{for all $j \in N \setminus \{\max(N)\}$, we have $e^j\cap e^{j+1}=\{y_j\}=\{x_{j+1}\}$, and}
    \item{$\sum_{j\in N}\Delta(e^j)\leq 2d_\Delta(x,y)$.}
\end{enumerate}
The existence of such a collection follows from the definition of $d_\Delta(x,y)$. We may assume that $x\in e^0$ and $y\in e^{\max(N)}$. Furthermore, by splitting $e^0$ into $e^0_0,e^0_1$ if necessary, we may assume that $0<\max(N)$ and $\sum_{j\in N}\Delta(e^j)\leq 4d_\Delta(x,y)$.  

We may also assume that $x\not=y_0$ (else we discard the edge $e^0$). If $x$ is contained in the interior of $e^0$, since $x\in \mathcal{V}$, there exists a finite sequence of dyadic edges $(\hat{e}^k)_{k\in M}$ such that $x=\hat{x}_0<\hat{x}_1<\dots<\hat{x}_{\max(M)}=y_0$ and, for each $k\in M$, we have $\hat{e}^k=[\hat{x}_{k-1},\hat{x}_k]\subset e^0$. Furthermore, each $\hat{e}^k$ is an $n_k$-edge for some $n_k$, where the sequence $(n_k)_{k\in M}$ is strictly decreasing. Here we note that obtaining the sequence $(\hat{x}_k)_{k\in M}$ is akin to obtaining the binary expansion of a dyadic number. 

It follows from the definition of a dyadic diameter function that, for $k\in M$, we have
\begin{equation}\label{E:max}
\Delta(\hat{e}^k)\leq\Delta(e^0).
\end{equation}
Since $([0,1],d_\Delta)$ is doubling, by \cite[Lemma 3.7]{HM12} there exists $n_0\in \mathbb{N}$ (depending only on the doubling constant) such that, for any dyadic $n$-edge $e$ and any dyadic $(n+n_0)$-edge $e'\subset e$, we have 
\begin{equation}\label{E:doubling}
\Delta(e')\leq \frac{1}{2}\Delta(e).
\end{equation}
Suppose $\max(M)\leq n_0$. Then, by (\ref{E:max}), we have
\[\sum_{k\in M}\Delta(\hat{e}^k)\leq \max(M)\Delta(e^0)\leq n_0\Delta(e^0).\]
On the other hand, suppose $\max(M)>n_0$. Given $\alpha>0$, write $\lfloor \alpha\rfloor$ to denote the greatest integer not greater than $\alpha$, and $\lceil \alpha\rceil$ to denote the least integer not less than $\alpha$. Using this notation, define $n_0':=\lceil \max(M)/n_0\rceil$. By (\ref{E:doubling}), we have
\[\sum_{k\in M}\Delta(\hat{e}^k)\leq \sum_{k\in M}2^{-\lfloor k/n_0\rfloor}\Delta(e^0)=n_0\Delta(e^0)\sum_{0\leq l\leq n_0'}2^{-l}\leq 2n_0\Delta(e^0).\]
Via a parallel argument applied to the edge $e^{\max(N)}$, we may assume $(e^j)_{j\in N}=([x_j,y_j])_{j\in N}$ is such that $x=x_0$, $y=y_{\max(N)}$, and $\sum_{j\in N}\Delta(e^j)\leq 8n_0d_\Delta(x,y)$.

Finally, we observe that
\begin{align*}
|f(x)-f(y)|\leq\sum_{j\in N}|f(x_j)-f(y_j)|&\leq\sum_{j\in N} Ld_\Delta(x_j,y_j)\\
&=L\sum_{j\in N}\Delta(e^j)\leq 8Ln_0d_\Delta(x,y).
\end{align*}
The conclusion follows with $L'=8n_0$.
\end{proof}

The closing lemma of this section is a technical result used to compare distances coming from different dyadic diameter functions. It is applied in the proof of Proposition \ref{prop:ddn}\eqref{item:ddn3}. 

\begin{lemma} \label{lem:d1=d2}
Suppose $\Delta_1,\Delta_2 \in \mathfrak{D}$ are two dyadic diameter functions. Let $L(\Delta_1,\Delta_2)$ (resp. $R(\Delta_1,\Delta_2)$) denote the closure of the set of points that are left (resp. right) endpoints of dyadic edges $e' \in \D$ with $\Delta_1(e) = \Delta_2(e)$ for all dyadic ancestors $e \supset e'$. Let $\mathcal{P}(\Delta_1,\Delta_2)$ denote the set of all intervals $[y,z]$ such that $y \in L(\Delta_1,\Delta_2)$ and $z \in R(\Delta_1,\Delta_2)$. Then $d_{\Delta_1}(y,z) = d_{\Delta_2}(y,z)$ for all $[y,z] \in \mathcal{P}(\Delta_1,\Delta_2)$.
\end{lemma}

\begin{proof}
Let $y,z$ be as in the statement of the lemma with $y < z$. By continuity of $d_{\Delta_1}$ and $d_{\Delta_2}$, it suffices to assume that $y$ is the left endpoint of a dyadic edge $e^{(y)}$ and $z$ is the right endpoint of a dyadic edge $e^{(z)}$ satisfying $\Delta_1(\tilde{e}) = \Delta_2(\tilde{e})$ for any dyadic edge $\tilde{e}$ that is an ancestor of $e^{(y)}$ or an ancestor of $e^{(z)}$.

Let $\Delta \in \mathfrak{D}$ be arbitrary. We claim that the definition of $d_\Delta(y,z)$ remains unchanged if we take the infimum over all dyadic chains $(e^i)_{i\in I}$ with $\{y,z\} \in \bigcup_{i\in I} e^i$ \emph{and}, for all $i \in I$,
\begin{align} \label{eq:d1=d2}
    &e^i \text{ equals, or is a dyadic sibling of,} \\
\nonumber    &\text{ an ancestor of } e^{(y)} \text{ or an ancestor of } e^{(z)}.
\end{align}
Applying this claim to $\Delta = \Delta_1$ and then to $\Delta = \Delta_2$, and noting that $\Delta(e_0)=\Delta(e_1)$ always holds for dyadic siblings $e_0,e_1$, we get that $d_{\Delta_1}(y,z) = d_{\Delta_2}(y,z)$. Thus it remains to verify this claim.

By \cite[Lemmas 3.2 and 3.3]{Freeman20}, the definition of $d_\Delta(y,z)$ remains unchanged if we take the infimum over all dyadic chains $(e^i)_{i\in I}$ with $\{y,z\} \in \bigcup_{i\in I} e^i$ satisfying the following properties.
\begin{enumerate}
    \item\label{item:d1=d2(1)} $\max(e^{i-1}) = \min(e^i)$ for all $1 \leq i \leq \max{I}$.
    \item\label{item:d1=d2(2)} $e^i$ is not a sibling of $e^j$ for any $i \neq j \in I$.
    \item\label{item:d1=d2(3)} There exists a unique $i^* \in I$ such that one of the following cases holds.
    \begin{enumerate}
        \item\label{item:d1=d2(3a)} The function $\gen: I \to \N$ is strictly decreasing on $\{0,\dots i^*\}$ and strictly increasing on $\{i^*,\dots \max(I)\}$.
        \item\label{item:d1=d2(3b)} The function $\gen: I \to \N$ is strictly decreasing on $\{0,\dots i^*\}$ and strictly increasing on $\{i^*+1,\dots \max(I)\}$, and $\gen(i^*) = \gen(i^*+1)$.
    \end{enumerate}
\end{enumerate}
Additionally, by discarding an initial and/or terminal segment of the chain, the definition of $d_\Delta(y,z)$ is also unchanged if we require
\begin{enumerate}[resume]
    \item\label{item:d1=d2(4)} $y \in e^0 \setminus \max(e^0)$ and $z \in e^{\max(I)} \setminus \min(e^{\max(I)})$.
\end{enumerate}

We will prove \eqref{eq:d1=d2} only for $i \in \{0,\dots i^*\}$ -- a symmetric argument covers the remaining case $i \in \{i^*+1,\dots \max(I)\}$. Let $j \in \{0,\dots i^*\}$ with $j \leq i^*-1$. We note that $e^j$ must be to the right of its dyadic sibling $e^j_{sib}$ and not to the left. Indeed, we otherwise have by \eqref{item:d1=d2(1)},\eqref{item:d1=d2(3)} that the intersection $(e^j \cup e^j_{sib}) \cap e^{j+1} = e^j_{sib}$ is equal to neither the empty set, a single point, $e^j\cup e^j_{sib}$, nor $e^{j+1}$. This is a contradiction since both $e^j \cup e^j_{sib}$ and $e^{j+1}$ are dyadic edges. Since $e^j_{sib}$ is to the left of $e^j$, \eqref{item:d1=d2(1)},\eqref{item:d1=d2(3)} imply $e^0 \subset e_{sib}^j$. Thus, if we can show $e^{(y)} \subset e^0$, then we get \eqref{eq:d1=d2} for $i=j$.

Assume towards a contradiction that $e^{(y)} \not\subset e^0$. If $y$ is contained in the interior of $e^0$, then it must happen that $e^{(y)} \subset e^0$, a contradiction. Thus, it must be the case that $y$ is not in the interior of $e^0$. Then by \eqref{item:d1=d2(4)}, $y = \min(e^0)$. Thus, the two dyadic edges $e^0,e^{(y)}$ have a common left endpoint (namely $y$), implying one is contained in the other. By assumption, $e^{(y)} \not\subset e^0$, and so $e^0 \subsetneq e^{(y)}$. But this implies $e^0$ is to the left of $e^0_{sib}$, contradicting the previous paragraph. Therefore, our assumption that $e^{(y)} \not\subset e^0$ is incorrect, and by the previous paragraph, \eqref{eq:d1=d2} holds for all $i \in \{0,\dots i^*-1\}$. It remains to prove that \eqref{eq:d1=d2} holds for $i = i^*$.

If $e^{i^*}$ is to the right of its sibling and not to the left, then as before \eqref{item:d1=d2(1)},\eqref{item:d1=d2(3)} and the fact that $e^{(y)} \subset e^0$ imply $e^{(y)} \subset e^0 \subset e_{sib}^{i^*}$, and we are done. Assume, then, that $e^{i^*}$ is to the left of its sibling. There are two cases to consider: either \eqref{item:d1=d2(3a)} holds or \eqref{item:d1=d2(3b)} holds. If \eqref{item:d1=d2(3b)} holds, then the assumption that $e^{i^*}$ is the left of its sibling implies $e^{i^*+1}$ is the sibling of $e^{i^*}$, contradicting \eqref{item:d1=d2(2)}. Hence, \eqref{item:d1=d2(3a)} must hold. Then \eqref{item:d1=d2(1)},\eqref{item:d1=d2(3a)} and the assumption that $e^{i^*}$ is to the left of $e_{sib}^{i^*}$ imply $e^{\max(I)} \subset e_{sib}^{i^*}$. Analogously as above, we also obtain $e^{(z)}\subset e^{\max(I)}$, and this in combination with $e^{\max(I)}\subset e_{sib}^{i^*}$ proves \eqref{eq:d1=d2} for $i = i^*$.
\end{proof}

\section{The Lipschitz Free Space of a QC Arc}\label{S:arc_freespace}

The goal of this section is to identify the linear isomorphism type of $\F(\gamma)$ when $\gamma$ is a QC arc. We achieve this in Corollary \ref{cor:isoL1}. As an added bonus, we are able to use intermediate results to deduce a rectifiable/purely unrectifiable decomposition (Theorem \ref{thm:rp1udecomp}) for QC arcs.

Since bi-Lipschitz homeomorphisms induce linear isomorphisms on the corresponding Lipschitz free spaces, Theorem \ref{T:HM} allows us to consider only QC arcs in $\mathcal{S}_1'$. Towards this end, we fix a dyadic diameter function $\Delta \in \mathfrak{D}$ such that $([0,1],d_\Delta)$ is doubling for the remainder of this section. Since $\Delta$ is fixed, we suppress notation and simply write ``$d$" and ``$\diam$" instead of ``$d_\Delta$" and ``$\diam_\Delta$". By passing to a 2-bi-Lipschitz equivalent metric, we may and do assume that $\Delta([0,\tfrac{1}{2}]) = \Delta([\tfrac{1}{2},1]) = \Delta([0,1]) = 1$.

Our first main conceptual tool is a filtration of $\sigma$-algebras on $[0,1]$ that selects sets on which the diameter $\diam$ is ``close to" the Euclidean diameter.

\begin{definition}[The Filtration]\label{D:filt}
Fix $n \geq 0$. Define $\A_n$ to be the $\sigma$-algebra on $[0,1]$ generated by Lebesgue-null sets and the collection of closed sets $\{e \in \D: \Delta(e) \leq 2^{n-\gen(e)}\}$. Note that $(\A_n)_{n\geq0}$ forms a filtration in the sense that $\A_n \subset \A_{n+1}$, and also note that $\A_0 = \{N,[0,1] \setminus N: N \text{ Lebesgue-null}\}$ by the last assumption in the previous paragraph.
\end{definition}

The next important concept is the distinction between the diffuse and atomic parts of the $\sigma$-algebras in the filtration.

\begin{definition}[Diffuse and Atomic Parts]\label{D:diff_atoms}
We decompose $\A_n$ into its atoms and its diffuse part, where the \emph{atomic part} is the collection of atoms $\At_n := \{e \in \D \cap \A_n: (\D\cap\A_n) \ni e' \subset e \Rightarrow e' = e\}$, and the \emph{diffuse part} is $\Diff_n := [0,1] \setminus \left(\bigcup_{e \in \At_n} \intr(e)\right)$. Here, $\intr(e)$ denotes the interior of $e$. For example, $\At_0 = \{[0,1]\}$ and $\Diff_0 = \{0,1\}$. Note that $\At_n$ is a collection of closed sets while $\Diff_n$ is itself a closed set.
\end{definition}

In the next proposition, we summarize, without proof, some immediate properties of the atomic and diffuse parts that will be used throughout this section.

\begin{proposition}[Basic Properties of Atomic and Diffuse Parts]
Let $n\geq0$. Then the following hold.
\begin{itemize}
    \item For each $e \in \At_n$, $e \setminus \intr(e)$ is Lebesgue-null and thus $\intr(e),\Diff_n \in \A_n$.
    \item The collection $\At_n \cup \{\Diff_n\}$ forms an essential partition of $[0,1]$ in the sense that $[0,1] = (\cup \At_n) \cup \Diff_n$ and for any $A,B \in \At_n \cup$ $\{\Diff_n\}$, either $A=B$ or $A \cap B$ is Lebesgue-null.
    \item If $e \in \At_n$, then $e \in \D$ and $\Delta(e_0) = \Delta(e_1) = \Delta(e) = 2^{n-\gen(e)}$.
    \item If $e \in \D \cap \A_n$, then $e \notin \At_n$ if and only if $e_i \in \A_n$ for some dyadic child $e_i \subset e$ (equivalently, for all dyadic children $e_i \subset e$).
\end{itemize}
\end{proposition}

Here we establish a useful characterization of dyadic edges not belonging to $\A_n$.

\begin{lemma} \label{lem:notAn}
Let $n\geq0$ and $e' \in \D$. If $e' \not\in \A_n$, then there exists $e \in \At_n$ such that $e' \subsetneq e$.
\end{lemma}

\begin{proof}
We prove the contrapositive by induction on $\gen(e')$. The base case $\gen(e') = 0$ obviously holds. Assume that, for some $k\geq1$, the lemma holds whenever $\gen(e') < k$. Let $e' \in \D$ with $\gen(e') = k$ and $e'$ is not a dyadic descendant of any $e \in \At_n$. Applying the inductive hypothesis to $\hat{e'}$, where $\hat{e'}$ is the dyadic parent of $e'$, we get that $\hat{e'} \in \A_n$. Also, since $\hat{e'}$ is not a dyadic descendant of any $e \in \At_n$, a fortiori $\hat{e'} \notin \At_n$. Then, since $\hat{e'} \in \A_n \setminus \At_n$, there must exist a proper dyadic descendant $f \subsetneq \hat{e'}$ with $f \in \A_n$, by definition of $\At_n$. Note that $f \subsetneq \hat{e'}$ implies $f \subset e'$ or $f \subset e'_{sib}$, where $e'_{sib}$ denotes the dyadic sibling of $e'$. It can be easily checked that if $e' \in \D \cap \A_n$ and $e \in \D$ is a dyadic ancestor of $e'$, then also $e \in \A_n$. Hence, it must hold that either $e' \in \A_n$ or $e'_{sib} \in \A_n$. It follows easily from the definition of $\A_n$ and the facts that $\Delta(e) = \Delta(e_{sib})$ and $\gen(e) = \gen(e_{sib})$ for any pair of dyadic siblings $e,e_{sib}$ that $e' \in \A_n$ if and only if $e'_{sib} \in \A_n$. Therefore, we must have $e' \in \A_n$. This completes the inductive step.
\end{proof}

The next two lemma flesh out some basic measure-theoretic properties of the diffuse parts.

\begin{lemma}[Restriction of Filtration to Diffuse Parts] \label{lem:LebDiffn}
For every $n\geq0$ and Lebesgue measurable set $A \subset [0,1]$, $A \cap \Diff_n \in \A_n$.
\end{lemma}

\begin{proof}
Let $n\geq0$. The collection $\{A \cap \Diff_n: A \subset [0,1] \text{ Lebesgue}\}$ is a $\sigma$-algebra generated by $\{E \cap \Diff_n: E \subset [0,1] \text{ Lebesgue-null or } E \in \D\}$. Hence, it suffices to show $E \cap \Diff_n \in \A_n$ whenever $E$ is Lebesgue-null or $E \in \D$. When $E$ is Lebesgue-null, so is $E \cap \Diff_n$, and thus it belongs to $\A_n$ by definition.

Assume, then, that $E \in \D$. By Lemma \ref{lem:notAn}, it must hold that $E \in \A_n$ or $E$ is a dyadic descendant of some $e \in \At_n$. Assume that the first case holds. Then since $\Diff_n$ is also in $\A_n$, we get $E \cap \Diff_n \in \A_n$, as needed. Now assume the second case holds, and let $e \in \At_n$ be a dyadic ancestor of $E$. Then $E \cap \Diff_n \subset E \setminus \intr(e)$ is a Lebesgue null set, and hence belongs to $\A_n$, as needed.
\end{proof}

\begin{lemma}[Atoms are Contained in Atoms] \label{lem:atom-in-atom}
For every $n\geq1$ and $e' \in \At_n$, there exists $e \in \At_{n-1}$ with $e \supset e'$.
\end{lemma}

\begin{proof}
Let $n\geq1$ and $e' \in \At_n$, and assume towards a contradiction that $e \not\supset e'$ for every $e \in \At_{n-1}$. Then by Lemma \ref{lem:notAn}, $e' \in \A_{n-1}$, which contradicts $e' \in \At_n$. 
\end{proof}

\begin{lemma}[Intervals with Endpoints in Diffuse Part] \label{lem:DiffnAn}
For every $n\geq0$ and $x,y \in \Diff_n$ with $x \leq y$, $[x,y] \in \A_n$.
\end{lemma}

\begin{proof}
Let $n\geq0$ and $x,y \in \Diff_n$ with $x \leq y$. Since $\{\Diff_n\} \cup \{\intr(e)\}_{e \in \At_n}$ forms a partition of $[0,1]$, $[x,y] = (\Diff_n \cap [x,y]) \cup \bigcup_{e \in \At_n} (\intr(e) \cap [x,y])$. Since $x,y \in \Diff_n$, either $\intr(e) \subset [x,y]$ or $\intr(e) \cap [x,y] = \emptyset$ for every $e \in \At_n$. Thus,
\begin{equation*} \label{eq:DiffnAn}
    [x,y] = (\Diff_n \cap [x,y]) \cup \bigcup_{e \in \At_n; e \subset [x,y]} \intr(e).
\end{equation*}
By Lemma \ref{lem:LebDiffn}, this shows $[x,y] \in \A_n$.
\end{proof}

\begin{remark}\label{rmk:Anmeasurability}
For $n\geq0$, Lemma \ref{lem:LebDiffn} implies that a function $g: [0,1] \to \R$ is $\A_n$-measurable if and only if it is Lebesgue-measurable and, for every $e \in \At_n$, $g\res_e$ equals a constant almost everywhere. 
\end{remark}

In the following lemma, we establish a technical point about the containment of dyadic edges in atomic intervals, which is needed in the proof of Theorem \ref{thm:rp1udecomp} (through Lemma \ref{lem:Up1u}). In particular, this technical point implies the density of $\bigcup_{n\geq0} \Diff_n$ in $[0,1]$, which will be used on multiple occasions throughout this section.

\begin{lemma}[Dyadic Edges Have Endpoints in a Diffuse Part] \label{lem:DDiff}
For every $[u,v] \in \D \setminus \{[0,1]\}$,  there exist $k\geq0$ and $e \in \At_k$ such that $[u,v] \subsetneq e$ and $\{u,v\} \subset \Diff_{k+1}$. Consequently, $\bigcup_{n\geq0} \Diff_n$ is dense in $[0,1]$.
\end{lemma}

\begin{proof}
Given the first sentence of the lemma, we note that the second sentence follows from the density of dyadic rationals in $[0,1]$.

Let $[u,v] \in \D \setminus \{[0,1]\}$, and choose $k' := \max\{n \geq 0: \exists \: e \in \At_n \text{ s.t. } [u,v]$ $\subset e\}$ (note that this maximum exists since $[0,1] \in \A_0$). Let $e' \in \At_{k'}$ be such that $[u,v]\subset e'$. If $[u,v] \subsetneq e'$, then it follows from the maximality of $k'$ that $\{u,v\}\subset \Diff_{k'+1}$ and thus $k := k'$ and $e := e'$ satisfies the conclusion. Assume, then, that $[u,v]=e'$. Since $[u,v] \neq [0,1]$, $k' \geq 1$. Then by Lemma \ref{lem:atom-in-atom}, there exists a dyadic ancestor $e \supset e'$ with $e \in \At_{k'-1}$. In this case $k := k'-1$ and $e$ also satisfy the desired conclusion.
\end{proof}

The next definitions are central to the remainder of the section. In them, we define a sequence of metrics $d_0 \leq d_1 \leq d_2 \leq \dots d$ such that $d_0$ is the Euclidean metric, each $d_n$ is bi-Lipschitz equivalent to $d_0$, and $d_n$ converges to $d$ as $n \to \infty$. The metrics $d_n$ are defined so as to agree with $d$ on $\Diff_n$ and to be a multiple of the Euclidean metric on each atom $e \in \At_n$ (Proposition \ref{prop:ddn}). Thus, we view $([0,1],d_n)$ as a rectifiable approximation of $([0,1],d)$ with controlled behavior on the intervals in $\A_n$. Being rectifiable, we may bring to bear the tools of 1-dimensional calculus to study their Lipschitz function theory (Lemma \ref{lem:Kirchheim}, Propositions \ref{prop:Lebesgue}, \ref{prop:ftc}, and \ref{prop:In}). 

\begin{definition}[Metrics and Measures]
Fix $n \geq 0$. We define a new dyadic diameter function $\Delta_n \in \mathfrak{D}$ by $\Delta_n(e') := 2^{\gen(e)-\gen(e')}\Delta(e)$ whenever $e' \subset e \in \At_n$ and $\Delta_n(e') := \Delta(e')$ otherwise. Let $\diam_n$ and $d_n$ denote the associated diameter and distance functions $\diam_{\Delta_n}$ and $d_{\Delta_n}$. Note that $d_n \leq d_{n+1} \leq d$ for every $n\geq0$. Let $\H^1_n$ denote the Hausdorff 1-measure with respect to the distance $d_n$.
\end{definition}

\begin{remark}\label{rmk:Delta=Delta_n}
By Lemma \ref{lem:notAn} and the fact that $\Delta(e) = 2^{n-\gen(e)}$ for $e \in \At_n$, we may equivalently define $\Delta_n$ by $\Delta_n(e') := 2^{n-\gen(e')}$ whenever $e' \not\in \A_n$ and $\Delta_n(e') := \Delta(e')$ whenever $e' \in \A_n$. From this it can be easily seen that $\{e \in \D: \Delta(e)=\Delta_n(e)\} = \A_n$.
\end{remark}

\begin{remark} \label{rmk:dnbiLipschitz}
From the previous remark, it's clear that, for every $n \geq 0$ and $x,y \in [0,1]$, $|x-y| \leq d_n(y,x) \leq 2^n|x-y|$. Hence, the $\sigma$-ideal of $\H_n^1$-null sets in $[0,1]$ coincides with the $\sigma$-ideal of Lebesgue-null sets, and in particular is independent of $n$. Thus, we may unambiguously qualify an event as happening \emph{almost everywhere} without reference to any specific measure $\H_n^1$.
\end{remark}

Because of the remarked upon bi-Lipschitz equivalence and the following lemma, the familiar fundamental theorems of Lebesgue calculus (Propositions \ref{prop:Lebesgue} and \ref{prop:ftc}) hold. Before getting to these, we recall three classical results of geometric measure theory. Throughout this section, whenever $(X,\A,\mu)$ is a measure space, $A \in \A$ with $0<\mu(A)<\
\infty$, and $f \in L^1(\mu)$, we use the notation $\fint_A f \, d\mu := \frac{1}{\mu(A)}\int_A f \, d\mu$ to denote the average value of $f$ over $A$.

\begin{lemma}[Metric and Measure Differentiation] \label{lem:Kirchheim}
Suppose $\rho$ is a metric on $[0,1]$ bi-Lipschitz equivalent to the Euclidean metric. Let $\H_\rho^1$ and $\diam_\rho$ denote the Hausdorff 1-measure and diameter functions with respect to $\rho$. Let $g \in L^\infty([0,1])$. Then for a.e. $x \in [0,1]$,
\begin{itemize}
    \item \cite[Theorem 1.8]{Heinonen} $\displaystyle{\lim_{y \to x^-}\fint_{[y,x]}g\:d\H_\rho^1 = g(x)}$,
    \item \cite[Theorem 2]{Kirchheim} the limit $\displaystyle{\lim_{y \to x}\dfrac{\rho(y,x)}{|x-y|}}$ exists in $(0,\infty)$, and
    \item \cite[Theorem 2.10.18(3)]{Federer} $\displaystyle{\limsup_{y \to x}\dfrac{\H_\rho^1([y,x])}{\diam_\rho([y,x])} \leq 1}$, and thus \\
    $\displaystyle{\lim_{y \to x}\dfrac{\H_\rho^1([y,x])}{\diam_\rho([y,x])} = 1}$ since $\H^1_{\rho}(A) \geq \diam_{\rho}(A)$ if $A$ is connected.
\end{itemize}
\end{lemma}

\begin{proposition}[Lebesgue Differentiation] \label{prop:Lebesgue}
Let $n \geq 0$ and $f \in$ \\ $\Lip_0([0,1],d_n)$. Then there exists a Lebesgue-null set $N \subset [0,1]$ such that the limit $\displaystyle{\lim_{y \to x^-}\dfrac{f(x)-f(y)}{d_n(y,x)}}$ exists for every $x \in [0,1] \setminus N$. Furthermore, the map
$\displaystyle{x \mapsto \lim_{y \to x^-}\dfrac{f(x)-f(y)}{d_n(y,x)}}$ defines a Lebesgue-measurable function on $[0,1] \setminus N$.
\end{proposition}

\begin{proof}
Let $n \geq 0$ and $f \in \Lip_0([0,1],d_n)$. Since $d_n$ is bi-Lipschitz equivalent to the Euclidean metric, $f$ is Lipschitz with respect to the Euclidean metric. Then by Rademacher's theorem, the limit
$$\lim_{y \to x^-}\dfrac{f(x)-f(y)}{|x-y|}$$
exists for a.e. $x \in [0,1]$. Then by Lemma \ref{lem:Kirchheim}, the limit
$$\lim_{y \to x^-}\dfrac{f(x)-f(y)}{d_n(y,x)} = \lim_{y \to x^-}\dfrac{f(x)-f(y)}{|x-y|}\cdot\lim_{y\to x}\dfrac{|x-y|}{d_n(y,x)}$$ exists
for a.e. $x \in [0,1]$. This proves the first part. The second part follows from basic measure theory (\cite[Chapter 1.4 Property 4]{Stein}) and the facts that $x \mapsto f(x)-f(y)$ and $x \mapsto d_n(y,x)$ are continuous for each fixed $y$.
\end{proof}

With the previous proposition in hand, we may safely define the derivative of a Lipschitz function in $\Lip_0([0,1],d_n)$.

\begin{definition}[Derivatives]
Fix $n \geq 0$ and $f \in \Lip_0([0,1],d_n)$. We define the \emph{derivative of $f$ with respect to $d_n$} by
    $$f^{(n)}(x) := \lim_{y \to x^-}\dfrac{f(x)-f(y)}{d_n(y,x)}$$
    at each $x \in [0,1]$ where the limit exists, and $f^{(n)}(x) :=0$ whenever the limit does not exist. We remark that the superscript ``$(n)$" indicates that a first order derivative (and not $n$th order derivative) is being taken with respect to the metric $d_n$. By Proposition \ref{prop:Lebesgue}, $f^{(n)}$ is a well-defined element of $L^\infty([0,1])$.
\end{definition}

\begin{proposition}[Fundamental Theorems of Calculus] \label{prop:ftc}
Let $n\geq0$, $f \in \Lip_0([0,1],d_n)$, and $g \in L^\infty([0,1])$. If we define $i_n(g): [0,1] \to \R$ by $i_n(g)(x) := \int_{[0,x]} g \, d\H_n^1$, then the following hold.
\begin{enumerate}
    \item\label{item:ftc2} $i_n(g) \in \Lip_0([0,1],d_n)$ and $i_n(g)^{(n)} = g$ almost everywhere.
    \item\label{item:ftc3} $\int_{[y,x]} f^{(n)} \: d\H_n^1 = i_n(f^{(n)})(x)-i_n(f^{(n)})(y)= f(x)-f(y)$ for every $[y,x] \subset [0,1]$.
\end{enumerate}
\end{proposition}

\begin{proof}
For \eqref{item:ftc2}, that $i_n(g) \in \Lip_0([0,1],d_n)$ follows easily from the Euclidean counterpart and the fact that $d_n$ is bi-Lipschitz equivalent to the Euclidean metric. Then we have, for a.e. $x \in [0,1]$,
\begin{align*}
    i_n(g)^{(n)}(x) &= \lim_{y \to x^-} \dfrac{i_n(g)(x)-i_n(g)(y)}{\diam_n([y,x])} \\
    &= \lim_{y \to x^-} \dfrac{1}{\diam_n([y,x])}\int_{[y,x]} g \, d\H_n^1 \\
    &\overset{\text{Lem }\ref{lem:Kirchheim}}{=} \lim_{y \to x^-} \fint_{[y,x]} g \, d\H_n^1 \\
    &\overset{\text{Lem }\ref{lem:Kirchheim}}{=} g(x).
\end{align*}

For \eqref{item:ftc3}, first consider an arbitrary Lipschitz function $h \in \Lip_0([0,1],d_n)$. By the proof of Proposition \ref{prop:Lebesgue}, $h$ is Lipschitz with respect to the Euclidean metric and the usual Euclidean derivative $h'$ is 0 almost everywhere if and only if $h^{(n)}$ is 0 almost everywhere. Hence, if $h^{(n)} = 0$ a.e., then $h$ is constant. We wish to apply this to the function $h$ defined by $h := i_n(f^{(n)}) - f$. We have by \eqref{item:ftc2} and linearity of the derivative that $h^{(n)} = 0$ a.e., and thus $i_n(f^{(n)})(y) - f(y) = i_n(f^{(n)})(x) - f(x)$ for every $[y,x] \subset [0,1]$, which proves \eqref{item:ftc3} after rearranging terms.
\end{proof}

The following lemma illustrates how $d_n$ behaves on the diffuse and atomic parts of $\A_n$.

\begin{proposition}[Restriction of $d_n$ to Atomic and Diffuse Parts] \label{prop:ddn}
Let $n\geq1$. Then the following hold.
\begin{enumerate}
    \item\label{item:ddn1} For every $e \in \At_{n-1}$, $e_i \in \{e_0,e_1\}$, and $[x,y] \subset e_i$, $d_{n}(x,y) = 2^n|y-x|$.
    \item\label{item:ddn2} For every $1 \leq k \leq n$ and $x,y \in e' \in \At_n$, $d_n(x,y) = 2^{n-k}d_k(x,y)$.
    \item\label{item:ddn3} For every $x,y \in \Diff_n$, $d_n(x,y) = d(x,y)$.
\end{enumerate}
\end{proposition}

\begin{proof}
Item \eqref{item:ddn1} follows quickly from the definitions of $\At_{n-1}$ and $d_n$. For \eqref{item:ddn2}, let $1 \leq k \leq n$ and $x,y \in e' \in \At_n$. By Lemma \ref{lem:atom-in-atom}, $e' \subset e_i$ for some $e \in \At_{n-1}$ and $e_i \in \{e_0,e_1\}$, and hence by \eqref{item:ddn1}
\begin{equation} \label{eq:ddn1}
    d_{n}(x,y) = 2^n|y-x|.
\end{equation}
By Lemma \ref{lem:atom-in-atom} and induction, $e' \subset \hat{e}_i$ for some $\hat{e} \in \At_{k-1}$ and $\hat{e}_i \in \{e_0,e_1\}$, and hence by \eqref{item:ddn1} again,
\begin{equation} \label{eq:ddn2}
    d_{k}(x,y) = 2^k|y-x|.
\end{equation}
Equations \eqref{eq:ddn1} and \eqref{eq:ddn2} imply $d_n(x,y) = 2^{n-k}d_k(x,y)$.

For \eqref{item:ddn3}, we claim that
\begin{equation*}
    \Diff_n \cap [0,1) \subset L(\Delta,\Delta_n) \:\:\:\text{ and }\:\:\: \Diff_n \cap (0,1] \subset R(\Delta,\Delta_n),
\end{equation*}
where $L(\Delta,\Delta_n)$ and $R(\Delta,\Delta_n)$ are defined as in Lemma \ref{lem:d1=d2}.
Once the above containments are proved, the conclusion follows immediately from Lemma \ref{lem:d1=d2}. Only the first containment will be verified, as the second follows from a similar argument. We will prove this by showing
\begin{equation} \label{eq:d=dn}
    [0,1) \setminus L(\Delta,\Delta_n) \subset \bigcup_{e \in \At_n} \intr(e).
\end{equation}

If the left-hand-side is empty, the containment is vacuous. Otherwise, let $x \in [0,1) \setminus L(\Delta,\Delta_n)$. For each dyadic edge $e \in \D$ with $e = [u,v]$, define $L(e) := [u,v)$. Then for each $k \geq 0$, $\{L(e)\}_{e \in \D_k}$ forms a partition of $[0,1)$ into half-open intervals. We then define $\D_k^n := \{L(e): e \in \D_k \text{ and } \Delta(e) = \Delta_n(e)\}$. Note that Remark \ref{rmk:Delta=Delta_n} implies
\begin{equation}\label{eq:ddn3}
    \D_k^n = \{L(e): e \in \D_k \cap \A_n\}.
\end{equation}
We make the following claim: \\
\underline{Claim}: $\bigcap_{k\geq1} \left(\cup \D_k^n\right) \subset L(\Delta,\Delta_n)$.

If $\bigcap_{k\geq1}\left(\cup\D_k^n\right)=\emptyset$, then the claim holds vacuously. Otherwise, let $w\in \bigcap_{k\geq1}\left(\cup\D_k^n\right)$. Then there exist intervals $\{L(e^k)\}_{k\geq1}$ with
\begin{itemize}
    \item $L(e^k) \in \D_k^n$ for all $k\geq1$,
    \item $e^{k}$ is the dyadic parent of $e^{k+1}$ for all $k\geq1$, and
    \item $\{w\}=\bigcap_{k\geq1} L(e^k)$.
\end{itemize}
Fix $k\geq1$, and let $\hat{e}$ be a dyadic ancestor of $e^k$. Then $\hat{e} = e^j$ for some $1 \leq j \leq k$, or $\hat{e} = [0,1]$. We will show that, in either case, $\Delta(\hat{e}) = \Delta_n(\hat{e})$. This obviously holds if $\hat{e} = [0,1]$, so assume the first case holds. Then  $L(\hat{e}) \in \D_j^n$ for some $j \geq 1$, and so $\Delta(\hat{e}) = \Delta_n(\hat{e})$ by definition of $\D_j^n$. Since $\hat{e}$ is an arbitrary dyadic ancestor of $e^k$, this shows $\min(e^k) \in L(\Delta,\Delta_n)$ by definition of $L(\Delta,\Delta_n)$. Since $\{w\}=\bigcap_{k\geq1} L(e^k)$, we have $w = \lim_{k \to \infty} \min(e^k)$, and thus $w \in L(\Delta,\Delta_n)$ since $L(\Delta,\Delta_n)$ is closed. This proves the claim.

By the claim, there exists $k\geq1$ such that $L(e^k_x) \notin \D_k^n$, where $L(e^k_x)$ is the unique element of $\{L(e)\}_{e \in \D_k}$ containing $x$. Let $k_*$ be the minimum of all such $k$. Then, letting $\hat{e}$ denote the dyadic parent of $e^{k_*}_x$ (which exists since $k_* \geq 1$), we must have that $\hat{e} \in \At_n$, because otherwise $e^{k_*}_x \in \A_n$ which contradicts $L(e^{k_*}_x) \notin \D_{k_*}^n$ by \eqref{eq:ddn3}. Furthermore, due to the minimality of $k_*$, it must hold that $L(\tilde{e})$ belongs to $\D_{j}^n$ whenever $j \leq k_*-1$ and $\tilde{e} \in \D_j$ is a dyadic ancestor of $\hat{e}$. Together with the fact that $x \notin L(\Delta,\Delta_n)$, this implies $x$ is not the left endpoint of $\hat{e}$. But since $x \in L(\hat{e})$, it must be that $x \in \intr(\hat{e})$. This proves \eqref{eq:d=dn}.
\end{proof}

Now that our main measure-theoretic objects have been established, we introduce the Banach spaces and operators that will help us determine the isomorphism type of $\F([0,1],d)$. We begin with the function spaces of Lebesgue.

\begin{definition}[$L^p$ spaces]
Fix $n \geq 0$. For $p \in [1,\infty]$, we write $L^p(\A_n)$ for the Banach space of (equivalence classes of) $p$-integrable functions over the measure space $([0,1],\A_n,\H_n^1)$. Note that, since each of the measure spaces have the same $\sigma$-ideal of null sets, the containment $L^\infty(\A_n) \subset L^\infty(\A_{n+1})$ holds isometrically.
\end{definition}

As usual, the filtration $(\A_n)_{n\geq0}$ affords us bounded linear maps $L^\infty(\A_{n})$ $\to L^\infty(\A_{n-1})$ known as conditional expectations. Here and throughout this section, $1_A$ denotes the indicator function of a set $A$.

\begin{definition}[Conditional Expectation]
Fix $n \geq 1$. The \emph{conditional expectation}, $\E^{n-1}: L^\infty(\A_{n}) \to L^\infty(\A_{n-1})$, is defined by
$$\E^{n-1}(g) := g1_{\Diff_{n-1}} + \sum_{e \in \At_{n-1}} \left(\fint_e g \: d\H^1_{n}\right)1_e.$$
Note that $\E^{n-1}(g)$ is indeed $\A_{n-1}$-measurable by Remark \ref{rmk:Anmeasurability}. We adopt the convention that $L^\infty(\A_{-1}) = \{1_\emptyset\}$ is the space consisting of only the constant function 0 and $\E^{-1}: L^\infty(\A_0) \to L^\infty(\A_{-1})$ is the 0 map.
\end{definition}

Notice that there is a slight peculiarity in the definition. The underlying measures $\H_{n}^1,\H_{n-1}^1$ in the domain and codomain of 
\[\E^{n-1}: L^\infty([0,1],\A_{n},\H_{n}^1) \to L^\infty([0,1],\A_{n-1},\H_{n-1}^1)\]
are different, and we have chosen to take the average of $g$ over each $e \in \At_{n-1}$ with respect to $\H_{n}^1$. However, due to Proposition~\ref{prop:ddn}\eqref{item:ddn1}, $\H^1_{n} \llcorner e = 2^{n}\L^1 \llcorner e$ for each $e \in \At_{n-1}$, where $\L^1$ is 1-dimensional Lebesgue measure. Thus, the definition remains unchanged if the average over each $e$ is taken with respect to $\L^1$, i.e.,
\begin{equation*} \label{eq:Hnaverage}
    \fint_e g \: d\H^1_{n} = \fint_e g \: d\L^1
\end{equation*}
for all $n \geq 1$, $g \in L^\infty(\A_{n})$, and $e \in \At_{n-1}$. In particular, we see that $\E^{n-1}$ is the usual conditional expectation with respect to Lebesgue measure, and thus the familiar properties of conditional expectation hold. We summarize them in the following proposition, omitting the standard measure-theoretic proof.

\begin{proposition} \label{prop:Eprops}
Let $n \geq 0$. Then the following hold.
\begin{enumerate}
    \item\label{item:Eprops1} For all $g \in L^\infty(\A_{n-1})$, $\E^{n-1}(g) = g$.
    \item\label{item:Eprops2} For all $0\leq k\leq n$, $g \in L^\infty(\A_n)$, and $g_{n-1}: [0,1] \to \R$, $g_{n-1} = \E^{n-1}(g)$ a.e. if and only if $g_{n-1} \in L^\infty(\A_{n-1})$ and $\int_{[0,1]} g_{n-1}h \, d\H_k^1 = \int_{[0,1]} gh \, d\H_k^1$ for all $h \in L^1(\A_{n-1})$.
    \item\label{item:Eprops3} $\E^{n-1} \circ \E^{n} = \E^{n-1}$.
    \item\label{item:Eprops4} $\E^{n-1}$ is linear, contractive, and weak*-weak* continuous.
\end{enumerate}
\end{proposition}

\begin{definition}[Martingale Difference Sequences]
A sequence $(D_n)_{n \geq 0}$ is a \emph{martingale difference sequence} if $D_n \in L^\infty(\A_n)$ and $\E^{n-1}(D_{n}) = 0$ for every $n \geq 0$. A martingale difference sequence $(D_n)_{n \geq 0}$ is $L^\infty$-\emph{bounded} if $\sup_{n \geq 0} \|D_n\|_\infty < \infty$. The set of $L^\infty$-bounded martingale difference sequences forms a weak*-closed linear subspace of the $\ell^\infty$-sum $\bigoplus^\infty_{n \geq 0} L^\infty(\A_n)$. It is precisely the subspace $\bigoplus^\infty_{n \geq 0} \ker(\E^{n-1})$, where $\ker(\E^{n-1})$ is the kernel of the conditional expectation $L^\infty(\A_{n}) \to L^\infty(\A_{n-1})$. Each space $\ker(\E^{n-1})$ is weak*-closed by Proposition \ref{prop:Eprops}\eqref{item:Eprops4}.
\end{definition}

\begin{example}\label{ex:mds}
By Proposition \ref{prop:Eprops}, for any sequence $(g_n)_{n\geq0} \in$ \\ $\bigoplus_{n\geq0}^\infty L^\infty(\A_n)$, the corresponding sequence $(g_n-\E^{n-1}(g_n))_{n\geq0}$ is an $L^\infty$-bounded martingale difference sequence with $\sup_{n \geq 0} \|g_n-\E^{n-1}(g_n)\|_{L^\infty} \leq 2\sup_{n \geq 0} \|g_n\|_{L^\infty}$.
\end{example}

Here we introduce the affinization of a function, which can intuitively be thought of as the ``integrated" version of conditional expectation. We make this intuition precise in Lemma \ref{lem:E^(n-1)}.

\begin{definition}[Affinization]
Fix $n\geq0$ and $f \in \Lip_0([0,1],d)$. Define the $n$th \emph{affinization} of $f$ to be the unique function $f_{\aff(n)}: [0,1] \to \R$ such that
\begin{itemize}
    \item $f_{\aff(n)}$ is continuous,
    \item $f_{\aff(n)}\res_{\Diff_n} = f\res_{\Diff_n}$, and
    \item $f_{\aff(n)}\res_e$ is affine for each $e \in \At_n$.
\end{itemize}
Obviously, $f \mapsto f_{\aff(n)}$ is linear. We also define $f_{\aff(-1)}$ to be the 0 function.
\end{definition}

Like conditional expectation, affinization enjoys the tower property. The proof is obvious and we omit it.

\begin{lemma} \label{lem:afftower}
For every $n\geq k\geq0$ and $f \in \Lip_0([0,1],d)$, we have $(f_{\aff(k)})_{\aff(n)} = f_{\aff(k)}$.
\end{lemma}

One of the key features of the $n$th affinization map is that it improves smoothness by transforming a $d$-Lipschitz function into a $d_n$-Lipschitz function, to which we can apply our calculus tools.

\begin{lemma}[$n$th Affinization is $d_n$-Lipschitz]\label{lem:affndnLip}
For every $n\geq0$ and $f \in \Lip_0([0,1],d)$, $f_{\aff(n)} \in \Lip_0([0,1],d_n)$ with $\|f_{\aff(n)}\|_{\Lip_0([0,1],d_n)} \leq$ $3\|f\|_{\Lip_0([0,1],d)}$. Moreover, if $f_i$ is a sequence in $\Lip_0([0,1],d)$ converging pointwise to some $f \in \Lip_0([0,1],d)$ as $i \to \infty$, then $(f_i)_{\aff(n)}$ converges pointwise to $f_{\aff(n)}$ as $i \to \infty$.
\end{lemma}

\begin{proof}
Let $n\geq0$ and $f \in \Lip_0([0,1],d)$. By Proposition \ref{prop:ddn}\eqref{item:ddn1}, $d_n$ is a multiple of the Euclidean metric on each $e \in \At_n$. Together with the fact $f_{\aff(n)}$ is affine on any $[u,v] \in \At_n$, this easily implies
$$\dfrac{|f_{\aff(n)}(x)-f_{\aff(n)}(y)|}{d_n(y,x)} = \dfrac{|f_{\aff(n)}(v)-f_{\aff(n)}(u)|}{d_n(u,v)}$$
for every $x \neq y \in [u,v]$. Since $u,v \in \Diff_n$,
$$f_{\aff(n)}(v)-f_{\aff(n)}(u) = f(v)-f(u),$$
and then combining these two equations gives us
\begin{align} \label{eq:Dn1}
    \dfrac{|f_{\aff(n)}(x)-f_{\aff(n)}(y)|}{d_n(y,x)} &= \dfrac{|f_{\aff(n)}(v)-f_{\aff(n)}(u)|}{d_n(u,v)} = \dfrac{|f(v)-f(u)|}{d_n(u,v)} \\ 
\nonumber    &\overset{\text{Prop } \ref{prop:ddn}\eqref{item:ddn3}}{=} \dfrac{|f(v)-f(u)|}{d(u,v)} \leq \|f\|_{\Lip_0([0,1],d)}.
\end{align}
Now let $[y,x] \subset [0,1]$ be arbitrary. There are three cases to consider: $x,y \in \bigcup_{e \in \At_n} \intr(e)$, $\{x,y\} \cap \Diff_n \neq \emptyset$ and $\{x,y\} \cap \bigcup_{e \in \At_n} \intr(e) \neq \emptyset$, and $x,y \in \Diff_n$. We will only treat the first case - the others yield no worse bounds. Let $[u_y,v_y],[u_x,v_x] \in \At_n$ such that $y \in (u_y,v_y)$ and $x \in (u_x,v_x)$. There are two subcases to consider: $v_y \leq u_x$ or $[u_y,v_y] = [u_x,v_x]$. Again, we treat the first case only as the second gives better bounds. Then we have
\begin{align*}
    &|f_{\aff(n)}(x)-f_{\aff(n)}(y)| \\
    & \leq |f_{\aff(n)}(v_y)-f_{\aff(n)}(y)| + |f_{\aff(n)}(u_x)-f_{\aff(n)}(v_y)| + |f_{\aff(n)}(x)-f(u_x)|  \\
    &\overset{\eqref{eq:Dn1}}{\leq} \|f\|_{\Lip_0([0,1],d)}d_n(y,v_y) + |f_{\aff(n)}(u_x)-f_{\aff(n)}(v_y)| + \|f\|_{\Lip_0([0,1],d)}d_n(u_x,x) \\
    &= \|f\|_{\Lip_0([0,1],d)}d_n(y,v_y) + |f(u_x)-f(v_y)| + \|f\|_{\Lip_0([0,1],d)}d_n(u_x,x) \\
    &\overset{\text{Prop }\ref{prop:ddn}\eqref{item:ddn3}}{\leq} \|f\|_{\Lip_0([0,1],d)}(d_n(y,v_y) + d_n(v_y,u_x) + d_n(u_x,x)) \\
    &\leq 3\|f\|_{\Lip_0([0,1],d)}d_n(y,x).
\end{align*}
The first equality above follows from the definiton of $f_{\aff(n)}$ and the fact that $u_x,v_y \in \Diff_n$, and the last inequality follows from the containments $[y,v_y],[v_y,u_x],[u_x,x] \subset [y,x]$ and the definition of $d_n$. This proves the first sentence in the case $x,y \in \bigcup_{e \in \At_n} \intr(e)$. The second and third cases can be treated similarly, and in fact we get the improved bounds $\|f_{\aff(n)}\|_{\Lip_0([0,1],d_n)} \leq 2\|f\|_{\Lip_0([0,1],d)}$ and $\|f_{\aff(n)}\|_{\Lip_0([0,1],d_n)} \leq \|f\|_{\Lip_0([0,1],d)}$, respectively. 

For the second sentence, let $x \in [0,1]$. If $x \in \Diff_n$, then
$$(f_i)_{\aff(n)}(x) = f_i(x) \overset{i\to\infty}{\to} f(x) = f_{\aff(n)}(x).$$
If $x = (1-t)u+tv \in [u,v] \in \At_n$, then $$(f_i)_{\aff(n)}(x) = (1-t)f_i(u) + tf_i(v) \overset{i\to\infty}{\to} (1-t)f(u)+tf(v) = f_{\aff(n)}(x).$$
\end{proof}

Lemma \ref{lem:affndnLip} and Proposition \ref{prop:Lebesgue} imply that the derivative $f_{\aff(n)}^{(n)}$ exists almost everywhere and is Lebesgue-measurable. We can say more.

\begin{lemma} \label{lem:Anmeasurability}
For every $n\geq0$ and $f \in \Lip_0([0,1],d)$, $f_{\aff(n)}^{(n)}$ is $\A_n$-measurable and $\|f_{\aff(n)}^{(n)}\|_{L^\infty(\A_n)} \leq \|f\|_{\Lip_0([0,1],d)}$.
\end{lemma}

\begin{proof}
Let $n\geq0$ and $f \in \Lip_0([0,1],d)$. By the preceding discussion $f_{\aff(n)}^{(n)}$ is Lebesgue-measurable, and thus by Remark \ref{rmk:Anmeasurability}, $f_{\aff(n)}^{(n)}$ is $\A_n$-measurable if, for every $e \in \At_n$, $f_{\aff(n)}^{(n)}\res_e$ equals a constant almost everywhere. Because $f_{\aff(n)}\res_e$ is, by definition, affine, and $d_n$ restricted to $e$ is a multiple of the Euclidean metric (Proposition \ref{prop:ddn}\eqref{item:ddn1}), it is obvious that
\begin{equation} \label{eq:Anmeasurability}
    f_{\aff(n)}^{(n)}(x) = \dfrac{f_{\aff(n)}(v)-f_{\aff(n)}(u)}{d_n(u,v)} \overset{\text{Prop }\ref{prop:ddn}\eqref{item:ddn3}}{=}\dfrac{f(v)-f(u)}{d(u,v)}
\end{equation}
for every $x \in \intr(e)$, where $e = [u,v]$. This proves $\A_n$-measurability.

For the second part, note that \eqref{eq:Anmeasurability} also proves that $|f_{\aff(n)}^{(n)}(x)| \leq$ \\ $\|f\|_{\Lip_0([0,1],d)}$ for a.e. $x \in \cup \At_n$, so it remains to show that $|f_{\aff(n)}^{(n)}(x)| \leq \|f\|_{\Lip_0([0,1],d)}$ for a.e. $x \in \cup \Diff_n$. Let $x \in \Diff_n$ such that $x$ is a left limit point of $\Diff_n$ and the limit defining $f_{\aff(n)}^{(n)}(x)$ exists. Note that by Proposition \ref{prop:Lebesgue} this constitutes a full measure subset of $\Diff_n$. Then we have
\begin{align*}
    |f_{\aff(n)}^{(n)}(x)| &= \lim_{y \to x^-} \dfrac{|f_{\aff(n)}(x)-f_{\aff(n)}(y)|}{d_n(y,x)} \\
    &= \lim_{\Diff_n \ni y \to x^-} \dfrac{|f_{\aff(n)}(x)-f_{\aff(n)}(y)|}{d_n(y,x)} \\
    &\overset{\text{Prop }\ref{prop:ddn}\eqref{item:ddn3}}{=} \lim_{\Diff_n \ni y \to x^-} \dfrac{|f(x)-f(y)|}{d(y,x)} \\
    &\leq \|f\|_{\Lip_0([0,1],d)}.
\end{align*}
\end{proof}

As another consequence of Lemma \ref{lem:affndnLip}, we get that the $n$th affinization operators converge to the identity as $n\to\infty$.

\begin{lemma}\label{lem:affconverge}
For every $f \in \Lip_0([0,1],d)$, $f_{\aff(n)}$ converges to $f$ pointwise as $n\to\infty$.
\end{lemma}

\begin{proof}
Let $f \in \Lip_0([0,1],d)$. By Lemma \ref{lem:affndnLip} and the fact that $d_n \leq d$ for every $n\geq0$, $\sup_{n\geq0} \|f_{\aff(n)}\|_{\Lip_0([0,1],d)} < \infty$. Hence, it suffices to prove pointwise convergence on the dense subset $\bigcup_{n\geq0} \Diff_n \subset [0,1]$. But this obviously happens since $f_{\aff(k)}(x) = f(x)$ whenever $x \in \Diff_n$ and $k \geq n$, by the definition of affinization and the fact that $\Diff_n \subset \Diff_k$.
\end{proof}

We are now ready to introduce the linear map $D$. We will see in Theorem \ref{thm:Diso} that $D$ takes $\Lip_0([0,1],d)$ isomorphically onto the space of $L^\infty$-bounded martingale difference sequences.

\begin{definition}[Derivative Martingale Difference Sequences]
For $n\geq0$, define the map $D_n: \Lip_0([0,1],d) \to \ker(\E^{n-1})$ by $D_n(f) := f_{\aff(n)}^{(n)} - \E^{n-1}(f_{\aff(n)}^{(n)})$. Note that, by Lemma \ref{lem:Anmeasurability} and Example \ref{ex:mds}, $D_n(f)$ indeed belongs to $\ker(\E^{n-1})$ and that $\|D_n(f)\|_{L^\infty(\A_n)} \leq 2 \|f\|_{\Lip_0([0,1],d)}$. Define $D: \Lip_0([0,1],d) \to \bigoplus_{n\geq0}^\infty \ker(\E^{n-1})$ by $D(f) := (D_n(f))_{n\geq0}$.
\end{definition}

\begin{theorem}[Weak*-Continuity of $D$] \label{thm:Dweak*}
$D$ is linear, 2-bounded, and weak*-weak*-continuous.
\end{theorem}

\begin{proof}
That $D$ is linear is clear, and that $D$ is 2-bounded follows from the comment in the definition of $D_n$. It remains to show weak*-weak*-continuity. By definition of the weak*-topology on $\ell^\infty$-sums, this happens if and only if $D_n$ is weak*-weak*-continuous for every $n\geq0$.

Let $n\geq0$. Since $\E^{n-1}$ and $f \mapsto f_{\aff(n)}$ are weak*-weak*-continuous (Proposition \ref{prop:Eprops}\eqref{item:Eprops4} and Lemma \ref{lem:affndnLip}), it remains to show that $f \mapsto f^{(n)}: \Lip_0([0,1],d_n) \to L^\infty([0,1],\H_n^1)$ is weak*-weak*-continuous.

Let $(f_i)_{i\geq0}$ be a sequence in $\Lip_0([0,1],d_n)$ with $\sup_{i\geq0} \|f_i\|_{\Lip([0,1],d_n)} < \infty$ and $f_i \to f_\infty$ pointwise for some $f_\infty \in \Lip_0([0,1],d_n)$. Let $h \in L^1([0,1],\H_n^1)$. Let $\eps>0$, and choose a smooth function $\phi: [0,1] \to \R$ with $\phi(0) = \phi(1) = 0$ and $\|h-\phi\|_{L^1([0,1],\H_n^1)} < \eps$. It can be readily verified using the standard proof of the Leibniz rule that, for $i^*\geq0$ or $i^*=\infty$, $(f_{i^*}\phi)^{(n)} = f_{i^*}^{(n)}\phi + f_{i^*}\phi^{(n)}$ almost everywhere. Together with Proposition \ref{prop:ftc}\eqref{item:ftc3} and the fact that $\phi(0)=\phi(1) = 0$, we get the integration-by-parts formula
\begin{equation}\label{eq:int-by-parts}
    \int_{[0,1]} f_{i^*}^{(n)}\phi \, d\H_n^1 = -\int_{[0,1]} f_{i^*}\phi^{(n)} \, d\H_n^1.
\end{equation}
Then we have, for $C:=\|f_\infty\|_{\Lip_0([0,1],d_n)}+\sup_{i\geq0}\|f_i\|_{\Lip_0([0,1],d_n)}$,

\begin{align*}
    \limsup_{i\to\infty}\left|\int_{[0,1]} (f_i^{(n)}-f^{(n)}_\infty)h \, d\H_n^1\right| &\leq C\eps + \limsup_{i\to\infty}\left|\int_{[0,1]} (f_i^{(n)}-f^{(n)}_\infty)\phi \, d\H_n^1\right| \\
    &\overset{\eqref{eq:int-by-parts}}{=} C\eps + \limsup_{i\to\infty}\left|\int_{[0,1]} (f_i-f_\infty)\phi^{(n)} \, d\H_n^1\right| \\
    &\overset{\text{DCT}}{=} C\eps.
\end{align*}
Since $C<\infty$ and $h \in L^1([0,1],\H_n^1)$ and $\eps>0$ were arbitrary, this shows $f_i^{(n)}$ weak*-converges to $f_\infty^{(n)}$ as $i \to \infty$.
\end{proof}

Here we give an alternate characterization of $D_n$ that will serve us in the proof of Theorem \ref{thm:Diso}.

\begin{lemma}\label{lem:E^(n-1)}
For every $n\geq0$ and $f \in \Lip_0([0,1],d)$, $\E^{n-1}(f_{\aff(n)}^{(n)}) = f_{\aff(n-1)}^{(n)}$ almost everywhere. Consequently, $D_n(f) = (f_{\aff(n)}-f_{\aff(n-1)})^{(n)}$.
\end{lemma}

\begin{proof}
The case $n=0$ is clear since both functions are 0 everywhere by definition. Let $n\geq1$ and $f \in \Lip_0([0,1],d)$. By definition of $\E^{n-1}$, we need to show that $f_{\aff(n)}^{(n)}1_{\Diff_{n-1}} = f_{\aff(n-1)}^{(n)}1_{\Diff_{n-1}}$ almost everywhere and, for every $e \in \At_{n-1}$, $\left(\fint_e f_{\aff(n)}^{(n)} \, d \H_n^1\right)1_e = f_{\aff(n-1)}^{(n)}1_{e}$ almost everywhere. This first equation is clear since
$$f_{\aff(n)}1_{\Diff_{n-1}} = f1_{\Diff_{n-1}}= f_{\aff(n-1)}1_{\Diff_{n-1}}$$
by definition of affinization. Now let $e = [u,v] \in \At_{n-1}$ with $e_0 = [u,m]$ and $e_1 = [m,v]$. Since $f_{\aff(n-1)}$ is affine on each $e$ and $d_n$ is a multiple of the Euclidean metric on each $e_i \in \{e_0,e_1\}$ (Proposition \ref{prop:ddn}\eqref{item:ddn1}), it holds that
\begin{equation} \label{eq:E^(n-1)1}
    \dfrac{f_{\aff(n-1)}(m)-f_{\aff(n-1)}(u)}{d_n(u,m)} = f_{\aff(n-1)}^{(n)}(x) = \dfrac{f_{\aff(n-1)}(v)-f_{\aff(n-1)}(m)}{d_n(m,v)}
\end{equation}
for every $x \in e \setminus \min(e)$ (which constitutes a co-null subset), and
\begin{equation} \label{eq:E^(n-1)2}
    d_n(u,m) = d_n(m,v) = \frac{1}{2}\H_n^1(e).
\end{equation}
Then we have
\begin{align*}
    \fint_e f_{\aff(n)}^{(n)} \, d \H_n^1 &\overset{\text{Prop }\ref{prop:ftc}\eqref{item:ftc3}}{=} \dfrac{f_{\aff(n)}(v)-f_{\aff(n)}(u)}{\H_n^1(e)} \\
    &= \dfrac{f(v)-f(u)}{\H_n^1(e)} \\
    &= \dfrac{f_{\aff(n-1)}(v)-f_{\aff(n-1)}(u)}{\H_n^1(e)} \\
    &= \dfrac{f_{\aff(n-1)}(v)-f_{\aff(n-1)}(m)}{\H_n^1(e)}+\dfrac{f_{\aff(n-1)}(m)-f_{\aff(n-1)}(u)}{\H_n^1(e)} \\
    &\overset{\eqref{eq:E^(n-1)2}}{=} \dfrac{f_{\aff(n-1)}(v)-f_{\aff(n-1)}(m)}{2d_n(u,m)}+\dfrac{f_{\aff(n-1)}(m)-f_{\aff(n-1)}(u)}{2d_n(m,v)} \\
    &\overset{\eqref{eq:E^(n-1)1}}{=}f_{\aff(n-1)}^{(n)}(x)
\end{align*}
for a.e. $x \in e$. This proves the first sentence. The second sentence follows from the first, the definition of $D_n$, and the linearity of the derivative.
\end{proof}

At this point we want to construct an inverse to $D$. Naturally, it should be an integral operators of sorts. We define individual integral operators $I_n$, investigate their essential properties, and then define the total integral operator $I$ and see that it inverts $D$.

\begin{definition}[$n$th Integrals]
Fix $n \geq 0$. Define the \emph{$n$th integral} $I_n: L^\infty(\A_n) \to \Lip_0([0,1],d)$ by
$$I_n(g)(x) := \int_{[0,x]} g \: d\H_n^1.$$
\end{definition}

\begin{proposition}[Basic Properties of $I_n$] \label{prop:In}
Let $n \geq 1$ and $g \in \ker(\E^{n-1})$. Then the following hold.
\begin{enumerate}
    \item\label{item:In1} $I_n(g) = I_n(g)_{\aff(n)}$.
    \item\label{item:In2} $I_n(g)$ vanishes on $\Diff_{n-1}$.
    \item\label{item:In3} $\|I_n(g)\|_{\Lip_0([0,1],d_n)} \leq 4\|g\|_{L^\infty(\A_n)}$.
\end{enumerate}
\end{proposition}

\begin{proof}
The first item follows easily from the facts that, for each $e \in \At_n$, $g$ is a.e. a constant on $e$ (by virtue of $\A_n$-measurability) and $d_n$ is a multiple of the Euclidean metric on $e$ (Proposition \ref{prop:ddn}\eqref{item:ddn1}).

For \eqref{item:In2}, let $x \in \Diff_{n-1}$. By Lemma \ref{lem:DiffnAn}, $[0,x] \in \A_{n-1}$. Together with the assumption that $g \in \ker(\E^{n-1})$, we get
\begin{align*}
    I_n(g)(x) = \int_{[0,x]} g \, \H_n^1 \overset{\text{Prop }\ref{prop:Eprops}\eqref{item:Eprops2}}{=} \int_{[0,x]} \E^{n-1}(g) \, \H_n^1 = \int_{[0,x]} 0 \, \H_n^1 = 0.
\end{align*}
Finally, we prove \eqref{item:In3}. Fix $e \in \At_{n-1}$ and $e_i \in \{e_0,e_1\}$. Set $f := I_n(g)$. By Proposition \ref{prop:ddn}\eqref{item:ddn1}, $d_n$ restricted to $e_i$ is a multiple of the Euclidean metric, which implies that, for every $[u,v] \in e_i$,
\begin{equation} \label{eq:In1}
    |f(v)-f(u)| \leq \|g\|_{L^\infty}d_n(u,v).
\end{equation}

Now let $[x,y] \subset [0,1]$ be arbitrary. As in the proof of Lemma \ref{lem:affndnLip}, there are three cases to consider: $x,y \in \bigcup_{e \in \At_{n-1}} \intr(e)$, $\{x,y\} \cap \Diff_{n-1} \neq \emptyset$ and $\{x,y\} \cap \bigcup_{e \in \At_{n-1}} \intr(e) \neq \emptyset$, and $x,y \in \Diff_{n-1}$. We will treat the first case only, the other cases follow from a similar argument (with no worse bounds). Assume $x,y \in \bigcup_{e \in \At_{n-1}} \intr(e)$. Let $[u_x,v_x],[u_y,v_y] \in \At_{n-1}$ such that $x \in (u_x,v_x)$ and $y \in (u_y,v_y)$. There are two more cases to consider: $v_x \leq u_y$ or $[u_x,v_x] = [u_y,v_y]$. We treat the first of these only - the second results in no worse bounds. By \eqref{item:In2},
\begin{equation} \label{eq:In2}
    f(v_x) = f(u_y) = 0.
\end{equation}
Let $m_x$ be the midpoint of $[u_x,v_x]$, so that $[u_x,v_x]_0 = [u_x,m_x]$ and $[u_x,v_x]_1 = [m_x,v_x]$. Similarly, let $m_y$ be the midpoint of $[u_y,v_y]$. There are four subcases to consdier: $x \leq m_x$ or $x \geq m_x$, and $y \leq m_y$ or $y \geq m_y$. We will only treat treat the case $x \leq m_x$ and $y \geq m_y$ - the other cases can be treated with a similar argument and yield possibly better bounds. Then we have
\begin{align*}
    |f(y)-f(x)| &\leq |f(m_x)-f(x)| + |f(v_x)-f(m_x)| + |f(u_y)-f(v_x)| \\
    &\:\:\:\:+ |f(u_y)-f(m_y)| + |f(y)-f(m_y)| \\
    &\overset{\eqref{eq:In2}}{=} |f(m_x)-f(x)| + |f(v_x)-f(m_x)| \\
    &\:\:\:\:+ |f(u_y)-f(m_y)| + |f(y)-f(m_y)| \\
    &\overset{\eqref{eq:In1}}{\leq} \|g\|_{L^\infty}d_n(x,m_x) + \|g\|_{L^\infty}d_n(m_x,v_x) \\
    &\:\:\:\:+ \|g\|_{L^\infty}d_n(m_y,u_y) + \|g\|_{L^\infty}d_n(m_y,y) \\
    &\leq 4\|g\|_{L^\infty}d_n(x,y). 
\end{align*}
Here, the final inequality relies on the fact that $([0,1],d_n)$ is $1$-bounded turning.
\end{proof}

\begin{definition}[Total Integral]
Define the \emph{total integral} map \\ $I:$ $\bigoplus^\infty_{n\geq 0} \ker(\E^{n-1}) \to \Lip_0([0,1],d)$ by $I := \sum_{n\geq 0} I_n$.
\end{definition}

The proof of the next theorem requires an application of Lemma \ref{L:dyadic}, and this is the one and only time when the doubling property is used.

\begin{theorem}[Boundedness and Weak*-Continuity of $I$] \label{thm:Ibndd}
The sum defining $I$ converges pointwise absolutely, $I$ is $8L'$-bounded, and $I$ is weak*-weak* continuous, where $L'$ is the constant from Lemma \ref{L:dyadic}.
\end{theorem}

\begin{proof}
Let $(g_n)_{n \geq 0} \in \bigoplus^\infty_{n\geq 0} \ker(\E^{n-1})$. We will control the Lipschitz constants of the partial sums $\sum_{n=0}^N I_n(g_n)$ when restricted to the endpoints of dyadic edges (independent of $N$ and the edge). From there we reach the desired conclusion with the help of Lemma \ref{L:dyadic}.

Set $A := \sup_{n \geq 0} \|g_n\|_{L^\infty} < \infty$. Fix $N \geq 0$, and let $f_N := \sum_{n=0}^N I_n(g_n)$. Obviously, $f_N$ is continuous. Let $[u,v] \in \D$. By Lemma \ref{lem:DDiff}, there exists $k\geq 0$ and $e \in \At_k$ such that $\{u,v\} \subset e \cap \Diff_{k+1}$. By Proposition \ref{prop:In}\eqref{item:In2}, this implies $I_n(g_n)(u) = I_n(g_n)(v) = 0$ for all $n \geq k+2$. Hence, if we define $M := \min\{N,k+1\}$, we have
\begin{equation} \label{eq:Ibndd1}
    f_N(v)-f_N(u) = \sum_{n=0}^M I_n(g_n)(v) - I_n(g_n)(u).
\end{equation}
Furthermore, by Proposition \ref{prop:In}\eqref{item:In3} and the definition of $A$, for each $0 \leq n \leq M$ we have
\begin{equation} \label{eq:Ibndd2}
    |I_n(g_n)(v) - I_n(g_n)(u)| \leq 4Ad_n(u,v).
\end{equation}
Combining these gives us
\begin{align*}
    |f_N(v)-f_N(u)| &\overset{\eqref{eq:Ibndd1}}{\leq} \sum_{n=0}^M |I_n(g_n)(v) - I_n(g_n)(u)| \overset{\eqref{eq:Ibndd2}}{\leq} 4A\sum_{n=0}^M d_n(u,v) \\
    &\overset{\text{Prop }\ref{prop:ddn}\eqref{item:ddn2}}{=} 4A\sum_{n=0}^{M} 2^{n-M}d_M(u,v) \leq 8Ad(u,v).
\end{align*}

Since $[u,v] \in \D$ was arbitrary, Lemma \ref{L:dyadic} implies that the Lipschitz constant of $f_N$ is bounded by $8L'A$. Since, by Lemma \ref{lem:DDiff} and Proposition \ref{prop:In}\eqref{item:In2}, $f_N(x)$ is eventually (in $N$) constant for $x$ in the dense subset $\bigcup_{n\geq 0} \Diff_n \subset [0,1]$, the first two claims follow.

For the third claim, let $g^i = (g^i_n)_{n\geq 0} \in \bigoplus^\infty_{n\geq 0} \ker(\E^{n-1})$ with \\ $\sup_i\sup_n\|g^i_n\|_{L^\infty}$ $< \infty$ and $(g^i_n)_{n\geq 0}$ weak*-converging to $g = (g_n)_{n\geq 0}$. This means that for every $(h_n)_{n \geq 0}$ with $h_n \in L^1(\A_n)$ and $\sum_{n\geq 0} \|h_n\|_{L^1} < \infty$,
\begin{equation} \label{eq:weak*def}
    \lim_{i \to \infty} \sum_{n\geq 0} \int g^i_nh_n \:d\H_n^1 = \sum_{n\geq 0} \int g_nh_n \:d\H_n^1.
\end{equation}
We need to show that $I(g^i)$ converges pointwise to $I(g)$. Since \[\sup_i \|I(g^i)\|_{\Lip_0([0,1],d)}<\infty \qquad \text{and} \qquad \|I(g)\|_{\Lip_0([0,1],d)} < \infty\]
(by boundedness of $I$), it suffices to prove pointwise convergence on the dense subset $\bigcup_{n\geq 0} \Diff_n \subset [0,1]$.

Fix $k \geq 0$ and $x \in \Diff_k$. Define $(h_n)_{n\geq 0}$ by $h_n = 1_{[0,x]}$ if $k \leq n$ and $h_n = 0$ if $n > k$. Then by \eqref{eq:weak*def}, we get
\begin{align*}
    \lim_{i\to\infty} I(g^i)(x)
    = \lim_{i\to\infty} \sum_{n\geq0} I_n(g^i_n)(x)
    = \lim_{i\to\infty} \sum_{n=0}^k I_n(g^i_n)(x) \\
    = \lim_{i\to\infty} \sum_{n=0}^k \int_{[0,x]}g^i_n \: d\H_n^1
    = \lim_{i\to\infty} \sum_{n\geq0} \int g^i_nh_n \: d\H_n^1 \\
    = \sum_{n\geq 0} \int g_nh_n \:d\H_n^1
    = \sum_{n=0}^k \int_{[0,x]}g_n \: d\H_n^1 \\
    = \sum_{n=0}^k I_n(g_n)(x)
    = \sum_{n\geq0} I_n(g_n)(x)
    = I(g)(x).
\end{align*}
\end{proof}

We conclude this subsection with our main theorem.

\begin{theorem}[Main Theorem] \label{thm:Diso}
The maps $D: \Lip_0([0,1],d) \to$ \\ $\bigoplus^\infty_{n\geq 0} \ker(\E^{n-1})$ and $I: \bigoplus^\infty_{n\geq 0} \ker(\E^{n-1}) \to \Lip_0([0,1],d)$ are inverses, and thus $D$ is an isomorphism. Moreover, the isomorphism constant depends only on the doubling constant of $([0,1],d)$.
\end{theorem}

\begin{proof}
First we prove that $D \circ I$ is the identity. By linearity and weak*-weak* continuity of $D$ and $I$ (Theorems \ref{thm:Dweak*} and \ref{thm:Ibndd}), and the fact that the weak*-closed linear span of $\bigcup_{n \geq 0} \ker(\E^{n-1})$ is all of $\bigoplus^\infty_{n \geq 0} \ker(\E^{n-1})$, it suffices to prove that $D \circ I$ is the identity when restricted to an arbitrary $\ker(\E^{k-1})$. Let $k \geq 0$ and $g = (g_n)_{n\geq 0} \in \bigoplus^\infty_{n \geq 0} \ker(\E^{n-1})$ with $g_n = 0$ for every $n \neq k$. Set $f := I(g) = I_k(g_k)$. Then by Proposition \ref{prop:In}\eqref{item:In1} and Lemma \ref{lem:afftower}, $f_{\aff(n)} = f$ for all $n \geq k$, and by Proposition \ref{prop:In}\eqref{item:In2}, $f$ vanishes on $\Diff_{k-1}$. It is easy to see by definition of affinization and the nesting property $\Diff_n \subset \Diff_{n+1}$ that this second fact implies $f_{\aff(n)} = 0$ for all $n < k$. The last two sentences together with Lemma \ref{lem:E^(n-1)} imply $D_k(f) = f^{(k)}$ and $D_n(f) = 0$ for all $n \neq k$. Then by Proposition \ref{prop:ftc}\eqref{item:ftc2} and the definition of $f$, $D_k(f) = g_k$ and $D_n(f) = 0$ for all $n \neq k$. That is, $D(I(g))=g$.

Now we prove that $I \circ D$ is the identity. Let $f \in \Lip_0([0,1],d)$ First note that Lemma \ref{lem:E^(n-1)} and Proposition \ref{prop:ftc}\eqref{item:ftc3} imply
\begin{equation} \label{eq:Diso1}
    I_n(D_n(f)) = f_{\aff(n)} - f_{\aff(n-1)}.
\end{equation}
Then by the definition of $I$ and Theorem \ref{thm:Ibndd},
\begin{equation} \label{eq:Diso2}
    I(D(f)) = \lim_{N\to\infty}\sum_{n=0}^N I_n(D_n(f))
\end{equation}
where the convergence is pointwise. Then we get
$$I(D(f)) \overset{\eqref{eq:Diso2}}{=} \lim_{N\to\infty}\sum_{n=0}^N I_n(D_n(f)) \overset{\eqref{eq:Diso1}}{=} \lim_{N\to\infty} f_{\aff(N)} \overset{\text{Lem }\ref{lem:affconverge}}{=} f.$$

The isomorphism constant of $D$ is the product of $\|D\|$ and $\|I\|$. By Theorem \ref{thm:Dweak*}, $\|D\| \leq 2$, and by Theorem \ref{thm:Ibndd}, $\|I\| \leq 8L'$, where $L'$ depends only on the doubling constant of $([0,1],d)$.
\end{proof}

\subsection{Isomorphism of Free Space to $L^1$-space}
In this subsection, we investigate the weak*-isomorphism type of the spaces $\ker(\E^{n-1})$ and use the results to identity the isomorphism type of $\F([0,1],d)$.

\begin{lemma} \label{lem:iso1}
For every $n \geq 1$, there exist finite measures $\mu_j$ on measurable spaces $Y_j$ such that $\ker(\E^{n-1})$ is weak*-isometric to $\bigoplus^\infty_j K_j$, where $K_j$ is the weak*-closed subspace of $L^\infty(\mu_j)$ consisting of all $g \in L^\infty(\mu_j)$ with $\int g \, d\mu_j = 0$.
\end{lemma}

\begin{proof}
Let $n\geq1$. For each $e \in \At_{n-1}$, let $Y_e$ be the measure space with underlying set $e$, underlying $\sigma$-algebra generated by the Lebesgue subsets of $\Diff_n \cap \,e$ and $\At_n \llcorner e$, and underlying measure $\H_n^1 \llcorner e$. Then we get a linear isometric embedding $\Phi: \ker(\E^{n-1}) \to \bigoplus^\infty_{e \in \At_{n-1}} L^\infty(Y_e)$ defined by $\Phi(g) := (g1_e)_{e \in \At_{n-1}}$. It is straightforward to verify from the definitions that $\Phi$ is linear, contractive, and weak*-weak*-continuous, and it is an isometric embedding since $g1_{\Diff_{n-1}} = 0$ a.e. for every $g \in \ker(\E^{n-1})$ and $[0,1] \setminus \Diff_{n-1} = \cup \At_{n-1}$ up to a Lebesgue-null set. By definition of $\ker(\E^{n-1})$, the image of $\Phi$ equals $\bigoplus^\infty_{e \in \At_{n-1}} K_e$, where $K_e = \{g \in L^\infty(\H_n^1 \llcorner e): \int g \, d\H_n^1 = 0\}$, proving the lemma.
\end{proof}

\begin{lemma} \label{lem:iso2}
For every $n \geq 0$, the predual space $(\ker(\E^{n-1}))_*$ is 128-isomorphic to $L^1(X_n)$ for some measure space $X_n$.
\end{lemma}

\begin{proof}
The case $n=0$ is trivial since $\ker(\E^{n-1})$ is 1-dimensional. Let $n \geq 1$. By Lemma \ref{lem:iso1}, $(\ker(\E^{n-1}))_*$ is isometric to $\bigoplus^1_j(K_j)_*$, where $K_j = \{g \in L^\infty(\mu_j): \int g \, d\mu_j = 0\}$ for some finite measures $\mu_j$ on measurable spaces $Y_j$. Since $(K_j)_*$ is isometric to $L^1(\mu_j)/\R1$, where $\R1 \subset L^1(\mu_j)$ denotes the constant functions, Lemma \ref{lem:L1/1R} implies $(K_j)_*$ is 128-isomorphic to $L^1(X_n^j)$ for some measure space $X_n^j$. Then $(\ker(\E^{n-1}))_*$ is 128-isomorphic to $L^1(X_n)$, where $X_n := \bigsqcup_j X_n^j$.
\end{proof}

Recall that a metric space is \emph{purely $k$-unrectifiable} if it contains no bi-Lipschitz copy of a positive measure subset of $\R^k$.

\begin{corollary}[Main Corollary] \label{cor:isoL1}
For every QC arc $\gamma$, $\F(\gamma)$ is isomorphic to $L^1(Z)$ for some measure space $Z$, where the isomorphism constant depends only on the bounded turning and doubling constants of $\gamma$. Moreover,  $Z$ is purely atomic if and only if $\gamma$ is purely 1-unrectifiable.
\end{corollary}

\begin{proof}
Let $(\gamma,\rho)$ be a $B$-bounded turning $QC$-arc. Since, for any scalar $c \in (0,\infty)$, $(\gamma, c\rho)$ is $B$-bounded turning and $\F(\gamma,c\rho)$ is isometric to $\F(\gamma,\rho)$, we may assume that $\diam(\gamma) = 1$. Then by Theorem \ref{T:HM}, $\gamma$ is $8B$-bi-Lipschitz equivalent to $([0,1],d)$ for some $([0,1],d) \in \mathcal{S}_1'$, and hence $\F(\gamma)$ is $8B$-isomorphic to $\F([0,1],d)$. By Theorem \ref{thm:Diso}, $\Lip_0([0,1],d)$ is weak*-isomorphic to $\bigoplus^\infty_{n\geq0} \ker(\E^{n-1})$, where the isomorphism constant depends only on the doubling constant of $([0,1],d)$ (which in turn depends only on $B$ and the doubling constant of $\gamma$). By predualizing, we get that $\F([0,1],d)$ is isomorphic to $\bigoplus^1_{n\geq0} (\ker(\E^{n-1}))_*$. Then by Lemma \ref{lem:iso2}, $\F([0,1],d)$ is 128-isomorphic to $\bigoplus^1_{n\geq0} L^1(X_n)$ for some measure spaces $X_n$. Since $\bigoplus^1_{n\geq0} L^1(X_n)$ is isometric to $L^1(Z)$ for $Z := \bigsqcup_{n\geq0} X_n$, the first sentence follows.

The second sentence follows from general Banach and Lipschitz free space theory. Recall that a Banach space has the \emph{Schur property} if every weakly-convergent sequence is norm-convergent, and note that this property is an isomorphic invariant. The Banach space $L^1(Z)$ has the Schur property if and only if $Z$ is purely atomic \cite[Section 4]{JL}, and by \cite[Theorem C]{AGPP}, $\F([0,1],d)$ has the Schur property if and only if $([0,1],d)$ is purely 1-unrectifiable.
\end{proof}

\subsection{Rectifiable/Purely Unrectifiable Decomposition} \label{SS:RP1U}
In this subsection, we no longer assume that $([0,1],d)$ is doubling. As mentioned before the statement of Theorem \ref{thm:Ibndd}, the only place where the doubling property is used is in the proof of Theorem \ref{thm:Ibndd} (through Lemma \ref{L:dyadic}). In particular, the results we use in this subsection (Lemma \ref{lem:DDiff}, Remark \ref{rmk:dnbiLipschitz}, and Propositions \ref{prop:ddn} and \ref{prop:In}) hold in this generality.

Recall that a metric space $X$ is \emph{countably $k$-rectifiable} if there exist countable collections of subsets $A_i \subset \R^k$ and Lipschitz maps $f_i: A_i \to X$ such that $\H^k\left(X \setminus \left(\bigcup_i f(A_i)\right)\right) = 0$, where $\H^k$ denotes the Hausdorff $k$-measure. By results from \cite{Kirchheim}, $X$ is purely $k$-unrectifiable if and only if every countably $k$-rectifiable subset is $\H^k$-null (see \cite[Section 1.3]{AGPP} for further explanation in the case $k=1$). 

There is a well-known decomposition theorem in geometric measure theory stating that any $\H^k$-$\sigma$-finite metric space is the union of a countably $k$-rectifiable subset and a purely $k$-unrectifiable subset (\cite[15.6 Theorem]{Mattila}). Although QC arcs need not be $\H^1$-$\sigma$-finite, the next theorem shows that they enjoy the same decomposition. Before proving the theorem, we review the definition of locally flat Lipschitz functions and their relationship to rectifiability.

A Lipschitz function $f: X \to \R$ on a metric space $(X,d)$ is \emph{locally flat} if
$$\lim_{x,y \to z} \dfrac{|f(x)-f(y)|}{d(x,y)} = 0$$
for every $z \in X$. We denote the vector space of locally flat Lipschitz functions by $\lip(X).$ We say that $\lip(X)$ \emph{separates points uniformly} if there exists $C < \infty$ such that for every $x,y \in X$, there is an $f \in \lip(X)$ with $\|f\|_{\Lip} \leq C$ and $d(x,y) \leq C|f(y)-f(x)|$. It follows from Lebesgue's density theorem and fundamental theorem of calculus that metric spaces whose points are uniformly separated by locally flat Lipschitz functions must be purely 1-unrectifiable (see the first paragraph of \cite[Section 2]{AGPP} for a discussion of the proof).

\begin{lemma} \label{lem:Up1u}
Let $U := [0,1] \setminus \bigcup_{n\geq0} \Diff_n$. Then $\lip(U)$ separates points uniformly. Consequently, $U$ is purely 1-unrectifiable.
\end{lemma}

\begin{proof}
We will prove the stronger statement that the Lipschitz functions locally flat on $U$ separate the points of $X$ uniformly, meaning that for every $x,y \in X$ there exists a $C$-Lipschitz function $f: X \to \R$ with $f\res_U \in \lip(U)$ and $d(x,y) \leq C|f(y)-f(x)|$. We claim that if a function $f$ is $C$-Lipschitz with respect to $d_k$ for some $k\geq0$, then $f$ is $C$-Lipschitz and locally flat on $U$ with respect to $d$. Indeed, the first part follows from the fact that $d_k \leq d$. The second part follows from the fact that $d_k$ is locally flat with respect to $d$ on $U$, meaning
$$\lim_{x,y \to z} \dfrac{d_k(x,y)}{d(x,y)} = 0$$
for every $z \in U$. To see this, let $z \in U$. Observe that, by definition of $U$ and $\Diff_n$, $U = \bigcap_{n\geq0} \bigcup_{e \in \At_n} \intr(e)$. Let $n \geq k$ be arbitrary. Then by the previous sentence we can find $e_{z,n} \in \At_n$ such that $z \in \intr(e_{z,n})$. Then we have
$$\lim_{x,y \to z} \dfrac{d_k(x,y)}{d(x,y)} \leq \sup_{x,y \in \intr(e_{z,n})} \dfrac{d_k(x,y)}{d(x,y)} \leq \sup_{x,y \in e_{z,n}} \dfrac{d_k(x,y)}{d_n(x,y)} \overset{\text{Prop }\ref{prop:ddn}\eqref{item:ddn2}}{=} 2^{k-n}.$$
Since $n\geq k$ was arbitrary, this proves
$$\lim_{x,y \to z} \dfrac{d_k(x,y)}{d(x,y)} = 0.$$
In conclusion, it suffices to find, for each $x,y \in [0,1]$, a function $f$ that is $C$-Lipschitz with respect to some $d_k$ satisfying $d(x,y) \leq C|f(y)-f(x)|$.

Let $x,y \in X$. By \cite[Lemma 3.5]{HM12}, there exists $[u,v] \in \D$ such that $[u,v] \subset [x,y]$ and $d(x,y) \leq 4d(u,v)$. Without loss of generality, we may assume $[u,v] \neq [0,1]$. Then it suffices to find a function $f$ that is $C$-Lipschitz with respect to some $d_k$ satisfying $d(u,v) \leq C|f(v)-f(u)|$ (this is because, by the definition of $d_k$, such a function $f$ can be redefined to satisfy $f(x) = f(u)$ and $f(y) = f(v)$ without increasing the Lipschitz constant). By Lemma \ref{lem:DDiff}, we can find $k\geq0$ and $e \in \At_k$ such that $[u,v] \subsetneq e$ and $\{u,v\} \subset \Diff_{k+1}$. Since $[u,v] \in \D$ and $[u,v] \subsetneq e$, it holds that either $[u,v] \subset e_0$ or $[u,v] \subset e_1$. Without loss of generality, assume $[u,v] \subset e_0$. Define $g: [0,1] \to \R$ by $g := 1_{e_0} - 1_{e_1}$. Since $e \in \A_k$, it is easy to see from the definition of $\A_k$ that $e_0,e_1 \in \A_{k+1}$, and thus $g \in L^\infty(\A_{k+1})$.
Additionally,
$$\E^k(g) = \left(\fint_e1_{e_0}-1_{e_1} \, d\H_{k+1}^1\right)1_e = \dfrac{\H_{k+1}^1(e_0) -\H_{k+1}^1(e_1)}{\H_{k+1}^1(e)} = 0.$$
Hence, $g \in \ker(\E^k)$. Let $f:=I_{k+1}(g)$. By definition of $I_{k+1}$ and $g$, we have
$$f(v)-f(u) = \int_{[u,v]}g\,d\H_{k+1}^1 = \H_{k+1}^1([u,v]) \geq d_{k+1}(u,v)  \overset{\text{Prop }\ref{prop:ddn}\eqref{item:ddn3}}{=} d(u,v).$$
By Proposition \ref{prop:In}\eqref{item:In3}, $f$ is 4-Lipschitz with respect to $d_k$.  This completes the proof.
\end{proof}

\begin{theorem}[Rectifiable/Purely Unrectifiable Decomposition] \label{thm:rp1udecomp}
For all bounded turning Jordan arcs $\gamma$, there exist $R,U \subset \gamma$ such that $R$ is countably 1-rectifiable, $U$ is purely 1-unrectifiable, and $\gamma = R \cup U$.
\end{theorem}

\begin{proof}
Since the conclusion of the corollary is invariant under bi-Lipschitz equivalences, it suffices to prove this when $\gamma = ([0,1],d)$, by Theorem \ref{T:HM}. By Remark \ref{rmk:dnbiLipschitz} and Proposition \ref{prop:ddn}\eqref{item:ddn3}, $R := \bigcup_{n\geq0} \Diff_n$ is countably 1-rectifiable. By Lemma \ref{lem:Up1u}, $U := [0,1] \setminus R$ is purely 1-unrectifiable.
\end{proof}

\section{Lipschitz Maps and Geometric Tree-Like Decompositions}\label{S:Lip_and_QA}

\subsection{Lipschitz Maps}
We begin this section by describing the relationship between Lipschitz functions on a metric space $X$ and Lipschitz functions on pieces of its geometric tree-like decomposition. In particular, the following result demonstrates that Lipschitz functions on the pieces of a geometric tree-like decomposition can be ``glued together'' to form a Lipschitz function on $X$.

Given a metric space $X$ with basepoint $x_0 \in X$ and a Banach space $\B$, we write $\Lip_0(X;\B)$ to denote the Banach space of Lipschitz maps $f:X \to \B$ (equipped with the Lipschitz norm) such that $f(x_0) = 0$.

\begin{theorem}\label{T:general}
Let $X$ be a metric space with $C$-geometric tree-like decomposition $\{X_n\}_{n\in N}$. Let $p_0$ be any point in $X_0$, and for $1 \leq n \in N$, let $p_n$ be the unique point in $X_n\cap\bigcup_{m<n}X_m$. For each $n \in N$, equip $X_n$ with basepoint $p_n$. Let $\B$ be a Banach space. Then the map $\Phi:\Lip_0(X;\B)\to\bigoplus_{n\in N}^\infty\Lip_0(X_n;\B)$ defined by $\Phi(f)_n:=f\res_{X_n}-f(p_n)$ is a linear $C$-isomor-\\phism. Moreover, if $\B = \R$, then $\Phi$ is weak*-weak* continuous.
\end{theorem}

\begin{proof}
As in the statement of the theorem, define $\Phi(f)_n:=f\res_{X_n}-f(p_n)$. It is clear that $\Phi(f)\in\bigoplus_{n\in N}^\infty\Lip_0(X_n)$, $\|\Phi\|\leq 1$, and that $\Phi_n$ preserves pointwise convergence (hence $\Phi$ is weak*-weak*-continuous when $\B = \R$).

Now we construct a linear map $\Psi: \bigoplus_{n\in N}^\infty\Lip_0(X_n;\B)\to\Lip_0(X;\B)$ that inverts $\Phi$ and has operator norm bounded by $C$. Let $f=(f_n)_{n\in N} \in\bigoplus_{n\in N}^\infty\Lip_0(X_n;\B)$. We will recursively define (over $m \in N$) functions $g_m: \bigcup_{n \leq m} X_n \to \B$ that extend one another, thus allowing us to define $\Psi(f)\res_{X_m}$ $:= g_m\res_{X_m}$.

For the base case, define $g_0 :=f_0$. For the inductive step, let $1 \leq m \in N$, and assume $g_{m-1}: \bigcup_{n \leq m-1} X_n \to \B$ has been defined. Then we define $g_m$ to agree with $g_{m-1}$ on $\bigcup_{n \leq m-1} X_n$ and to equal $g_{m-1}(p_m) + f_m$ on $X_m$. Since these two domains intersect exactly at the point $p_m$, $g_m$ is well-defined if and only if $g_{m-1}(p_m) = g_{m-1}(p_m) + f_m(p_m)$. This holds since $f_m \in \Lip_0(X_m;\B)$ and $p_m$ is the basepoint of $X_m$. This completes the recursive definition.

It is straightforward to check that $\Psi$ inverts $\Phi$, and we omit those details. It remains to show that $\Psi$ is bounded. We will show that $\Psi(f)$ is Lipschitz on $X$ with Lipschitz constant no greater than $C\sup_{n \in N}\|f_n\|_{\Lip}$, where $C$ is the geometric constant of the tree-like decomposition $X=\bigcup_{n\in N}X_n$. 

Fix points $x,y\in X$, and let $(z_i)_{i\in I}$ denote a short decomposition path from $x$ to $y$ such that, for each $1\leq i\leq\max(I)$, there exists $n_i\in N$ for which $\{z_{i-1},z_i\}\subset X_{n_i}$. We then get the estimate 
\begin{align*}
\|\Psi(f)(x)-\Psi(f)(y)\|&\leq\sum_{1\leq i\leq \max(I)}\|\Psi(f)(z_{i-1})-\Psi(f)(z_i)\|\\
&=\sum_{1\leq i\leq \max(I)}\|f_{n_i}(z_{i-1})-f_{n_i}(z_i)\|\\
&\leq\sup_{n \in N}\|f_n\|_{\Lip}\sum_{1\leq i\leq \max(I)}d(z_{i-1},z_i)\\
&\leq C\sup_{n \in N}\|f_n\|_{\Lip}d(x,y),
\end{align*}
where in the last inequality, we've used the definition of short decomposition path in a $C$-geometric tree-like decomposition.
\end{proof}

\begin{remark}
Clearly, item (2) from the definition of $C$-geometric tree-like decompositions is not used in the proof. Hence, Theorem \ref{T:general} holds under the weaker assumption that $\{X_n\}_{n \in N}$ satisfies items (1) and (3) in the definition only.
\end{remark}

\subsection{Lipschitz Light Maps} In this subsection, we investigate the relationship between Lipschitz light maps on pieces of a geometric tree-like decomposition of $X$ and Lipschitz light maps on $X$ itself. As with Lipschitz maps (see Theorem \ref{T:general}), we demonstrate how Lipschitz light maps on pieces of a geometric tree-like decomposition can be ``glued together'' to form a globally Lipschitz light map. We then point out how this result can be used to bound the Lipschitz dimension of a space.

Given a metric space $X$, a subset $E\subset X$, and a number $\varepsilon>0$, we define 
\[N(E;\varepsilon)=\{x\in X\,|\,\dist(x,E)<\varepsilon\}.\]

\begin{lemma}\label{L:gen_chain}
Suppose $X$ is a metric space with $C$-geometric tree-like decomposition $\{X_n\}_{n\in N}$. For any $x,y\in X$, $\delta>0$, and $\delta$-chain $(z_i)_{i\in I}$ from $x$ to $y$, there exists a $C\delta$-chain $(w_j)_{j\in J}$ from $x$ to $y$ and a subindexing set $J' \subset J$ such that
\begin{enumerate}
    \item{$\{w_{j}\}_{j\in J}\subset N(\{z_i\}_{i\in I},C\delta)$,}
    \item $(w_{j_k})_{k\in J'}$ is a minimal decomposition path from $x$ to $y$, and
    \item{for $1\leq k\leq\max(J')$, there exists $n_k\in N$ such that $\{w_j\,|\,j_{k-1}\leq j\leq j_k\}\subset X_{n_k}$.}
\end{enumerate}
\end{lemma}

\begin{proof}
Set $X_{m_0}$ to be an element of $\{X_n\}_{n\in N}$ containing $z_{i_0}:=z_0=x$. If $\{z_i\}_{i\in I}\subset X_{m_0}$, then it's clear that $(z_i)_{i\in I}$ satisfies the conclusions of the lemma, with $(w_{j_k})_{k\in J'}=(x,y)$. If $\{z_i\}_{i\in I}\not\subset X_{m_0}$, define $z_{i_1}$ to be the first element of $(z_i)_{i\in I}$ not contained in $X_{m_0}$. Write $X_{m_1}$ to denote an element of $\{X_n\}_{n\in N}$ containing $z_{i_1}$. Inductively, define $z_{i_l}$ to be the first element of $(z_i\,|\,i_{l-1}< i \leq \max(I))$  not contained in $X_{m_{l-1}}$, and write $X_{m_l}$ to denote an element of $\{X_n\}_{n\in N}$ containing $z_{i_l}$. In this way we obtain $\{X_{m_l}\}_{l\in L}$ such that $\{z_i\}_{i_l\leq i<i_{l+1}}\subset X_{m_l}$. Note also that
\[\{z_i\,|\,i_{\max(L)}\leq i\leq \max(I)\}\subset X_{m_{\max(L)}}.\]

Since $\{X_n\}_{n\in N}$ constitutes a $C$-geometric tree-like decomposition of $X$, for each $1\leq l \leq\max(L)$ there exists a short decomposition path $(v^l_m)_{m\in M_l}$ from $z_{i_{l}-1}$ to $z_{i_l}$. Moroever,
\begin{equation}\label{E:close}
\diam\left(\{v_m^l\}_{m\in M_l}\right)\leq Cd(z_{{i_l}-1},z_{i_l})\leq C\delta.
\end{equation}

For each $0\leq l\leq \max(L)-1$, define 
\[(u^l_k)_{k\in K_l}=(z_i\,|\,i_l\leq i< i_{l+1}).\]
Also, define 
\[(u_k^{\max(L)})_{k\in K_{\max(L)}}=(z_i\,|\,i_{\max(L)}\leq i\leq \max(I)).\]
Intuitively, each sequence in  $\mathcal{K}:=\{(u_k^l)_{k\in K_l}\,|\,l\in L\}$ consists of consecutive points from $(z_i)_{i\in I}$ contained in a single piece of the geometric tree-like decomposition, while each sequence in $\mathcal{M}:=\{(v_m^l)_{m\in M_l}\,|\,1\leq l\leq \max(L)\}$ is a short decomposition path joining consecutive sequences in $\mathcal{K}$.

We form $(w_j)_{j\in J}$ out of $\mathcal{K}$ and $\mathcal{M}$ as follows. For each $l\in L$, define the index set $J_{l,1}$ such that $(w^{l,1}_j)_{j\in J_{l,1}}=(u_k^l)_{k\in K_l}$. For each $1\leq l\leq\max(L)$, define the index set $J_{l,0}$ such that $(w^{l,0}_j)_{j\in J_{l,0}}=(v^l_m)_{m\in M_l}$. Define 
\begin{align*}
(w_j)_{j\in J}=(w^{0,1}_j)_{j\in J_{0,1}}&*(w^{1,0}_j)_{j\in J_{1,0}}*(w^{1,1}_j)_{j\in J_{1,1}}*(w^{2,0}_j)_{j\in J_{2,0}}*(w^{2,1}_j)_{j\in J_{2,1}}\dots\\
&* (w^{\max(L),0}_j)_{j\in J_{{\max(L)},0}}*(w^{\max(L),1}_j)_{j\in J_{{\max(L)},1}}.
\end{align*}

We now ensure that $(w_j)_{j\in J}$ satisfies the conclusions of the lemma. First, by (\ref{E:close}), $(w_j)_{j\in J}$ is a $C\delta$-chain. Again by (\ref{E:close}), the sequence $(w_j)_{j\in J}$ satisfies (1) in the statement of the lemma. Next, we note that the sequence 
\[(w_{j_k})_{k\in J'}:=(x)\hat{*}(w^{1,0}_j)_{j\in J_{1,0}}\hat{*}(w^{2,0}_j)_{j\in J_{2,0}}\hat{*}\dots\hat{*} (w^{\max(L),0}_j)_{j\in J_{{\max(L)},0}}\hat{*}(y)\]
is a decomposition path from $x$ to $y$ such that $\{w_{j_k}\}_{k\in J'}\subset \{w_j\}_{j\in J}$. While this decomposition path need not be minimal, we can nevertheless verify (3) in the statement of the lemma as follows: Let $1\leq k'\leq\max(J')$ be fixed. On one hand, it may be the case that $\{w_j\,|\,j_{k'-1}\leq j\leq j_{k'}\}$ consists of only two subsequent points in the decomposition path $(w_{j_k})_{k\in J'}$. On the other hand, it may be the case that
\[\{w_j\,|\,j_{k'-1}\leq j\leq j_{k'}\}=\{w_{j_{k'-1}},w_{j_{k'}}\}\cup\{w_j^{l,1}\}_{j\in J_{l,1}}=\{w_{j_{k'-1}},w_{j_{k'}}\}\cup\{u_k^l\}_{k\in K_l}\]
for some $l\in L$. In either case,  $\{w_j\,|\,j_{k'-1}\leq j\leq j_{k'}\}\subset X_{n_{k'}}$ for some $n_{k'}\in N$.

Finally, suppose the decomposition path $(w_{j_k})_{k\in J'}$ is not minimal. For ease of notation, for each $k\in J'$ we write $x_k:=w_{j_k}$. By Remark \ref{R:minimal}, there exists some $k_0<k_1$ such that $x_{k_1}\in X_{n_{k_0}}$, where $X_{n_{k_0}}$ is the unique element of $\{X_n\}_{n\in N}$ containing $\{x_{k_0-1},x_{k_0}\}$, and $k_1$ is the minimal index greater than $k_0$ for which this occurs. Given such an index $k_1\in J'$, we describe a \textit{Pruning Procedure} that can be applied to $(w_j)_{j\in J}$ while not impacting (1), (3), or the fact that $(w_j)_{j\in J}$ is a $C\delta$-chain. 

If $x_{k_1}=x_{k_0}$, then we simply delete $(w_{j_{k_0+1}},\dots,w_{j_{k_1}})$ from $(w_j)_{j\in J}$ (and thus $w_{j_{k_1}}$ from $(w_{j_k})_{k\in J'}$). Clearly, this deletion does not affect $(1)$ or $(3)$. It also does not change the fact that $(w_j)_{j\in J}$ is a $C\delta$-chain. We refer to this deletion as \textit{Pruning Procedure A}. 

Assume now that $x_{k_1}\not=x_{k_0}$. By the assumption that $k_1$ is minimal, the decomposition path $(x_{k_0},\dots,x_{k_1-1})$ is minimal and, for $k_0\leq l,k\leq k_1-1$, we have $x_l\not=x_k$ if $l\not=k$. Furthermore, $X_{n_l}\not=X_{n_k}$ if $l\not=k$. Therefore, $(x_{k_0},\dots,x_{k_1},x_{k_0})$ is a simple decomposition loop. By Lemma \ref{L:unique}, we must have $k_1=k_0+1$ (that is, $(x_{k_0},\dots,x_{k_1},x_{k_0})$ must be a \textit{trivial} simple decomposition loop), and thus
\begin{equation}\label{E:same_pieces}
X_{n_{k_0+1}}=X_{n_{k_0}}.
\end{equation}
Therefore, we simply delete $k_0$ from the index set $J'$. This leaves $(w_j)_{j\in J}$ unchanged, and therefore does not impact (1) or the fact that $(w_j)_{j\in J}$ is a $C\delta$-chain. Furthermore, since (\ref{E:same_pieces}) implies that $\{w_j\,|\,j_{k_0-1}\leq j\leq j_{k_0+1}\}\subset X_{n_{k_0+1}}$, neither do we affect (3). We refer to the deletion of $k_0$ from $J'$ as \textit{Pruning Procedure B}. 

Since an application of the above Pruning Procedure (version A or B) decreases the cardinality of $J'$ by at least one, the Pruning Procedure can be applied only finitely many times. Since $(w_{j_k})_{k\in J'}$ is minimal if the Pruning Procedure cannot be applied to $(w_j)_{j\in J}$, we can obtain a $C\delta$-chain satisfying (1)-(3).
\end{proof}

\begin{theorem}\label{T:general_LL}
Suppose $d\geq 1$ and $X$ is a metric space with $C$-geometric tree-like decomposition $\{X_n\}_{n\in N}$ and $\mathbb{B}$ is a Banach space. If there exist numbers $L,Q\geq 1$ such that, for each $n\in N$, there exists an $L$-Lipschitz and $Q$-light map $f_n:X_n\to\mathbb{B}$, then there exist $Q'\geq1$ and an $LC$-Lipschitz $Q'$-light map $f:X\to \mathbb{B}$. Here, $Q'$ depends only on $L$, $Q$, and $C$.
\end{theorem}

\begin{proof}
By Theorem \ref{T:general}, there exists an $LC$-Lipschitz map $f:X\to\mathbb{B}$ that restricts to a translation of $f_n$ on each subset $X_n$. Then our assumption that $f_n$ is $L$-Lipschitz $Q$-light implies that $f\res_{X_n}$ is $L$-Lipschitz $Q$-light for every $n \in N$. We next verify that $f$ is Lipschitz light on $X$. To this end, let $\delta > 0$ and fix a set $E\subset\mathbb{B}$ such that $\diam(E) < \delta$. Let $U\subset X$ be a $\delta$-component of $f^{-1}(E)$ in $X$. Let $x,y \in U$ be arbitrary and $(z_i)_{i\in I}$ a $\delta$-chain in $U$ from $x$ to $y$.

By Lemma \ref{L:gen_chain}, there exists a $C\delta$-chain $(w_j)_{j\in J}$ from $x$ to $y$ and a subindexing set $J' \subset J$ such that 
\begin{enumerate}[(i)]
    \item{$\{w_{j}\}_{j\in J}\subset N(\{z_i\}_{i\in I},C\delta)$,}
    \item $(w_{j_k})_{k\in J'}$ is a minimal decomposition path from $x$ to $y$, and
    \item{for $1\leq k\leq\max(J')$, there exists $n_k\in N$ such that $\{w_j\,|\,j_{k-1}\leq j\leq j_k\}\subset X_{n_k}$.} 
\end{enumerate}


Define $k_0\in J'$ such that 
\[d(w_{j_{k_0-1}},w_{j_{k_0}})=\max_{1\leq k\leq \max(J')}d(w_{j_{k-1}},w_{j_k}).\]
Since $(w_{j_k})_{k\in J'}$ is a minimal decomposition path from $x$ to $y$, we have
\begin{equation}\label{E:upper}
d(x,y)\leq C d(w_{j_{k_0-1}},w_{j_{k_0}}).
\end{equation}
Since $f$ is $LC$-Lipschitz, item (i) implies that $\{w_j\}_{j\in J}\subset f^{-1}(E')$, where $E'=N(E;LC^2\delta)$. Set $\delta' := \diam(E)+2LC^2\delta \geq \diam(E')$. Since $C,L\geq1$ and $\delta>\diam(E)$, we have
\begin{equation*}
    C\delta < \diam(E)+2LC^2\delta < (1+2LC^2)\delta,
\end{equation*}
and thus $\delta' = \diam(E)+2LC^2\delta$ gives us
\begin{equation}\label{eq:chain}
    C\delta < \delta' < (1+2LC^2)\delta.
\end{equation}
Since the sequence $(w_j\,|\,j_{k_0-1}\leq j\leq j_{k_0})$ is a $C\delta$-chain, \eqref{eq:chain} and (iii) imply that it is also a $\delta'$-chain in $\left(f\res_{X_{n_{k_0}}}\right)^{-1}(E')$. Since $f\res_{X_{n_{k_0}}}$ is $Q$-light, we conclude 
\[\diam(\{w_j\,|\,j_{k_0}\leq j\leq j_{k_0}\})\leq Q\delta'.\]
Via \eqref{E:upper}, we find that 
\begin{align*}
d(x,y)&\leq Cd(w_{j_{k_0-1}},w_{j_{k_0}})\leq C\diam(\{w_j\,|\,j_{k_0-1}\leq j\leq j_{k_0}\})\leq CQ\delta'.
\end{align*}
Since $x,y \in U$ were arbitrary, it follows from the above inequality and \eqref{eq:chain} that $\diam(U)\leq CQ(1+2LC^2)\delta$, and so $f:X\to\mathbb{B}$ is $LC$-Lipschitz and $Q'$-light, where $Q'=CQ(1+2LC^2)$.
\end{proof}

\begin{corollary}\label{C:Lip_dim_cor}
Suppose $d\geq1$ and $X$ is a metric space with $C$-geometric tree-like decomposition $\{X_n\}_{n\in N}$. If there exist $L,Q\geq1$ such that, for each $n\in N$, there exists a $L$-Lipschitz $Q$-light map $f_n:X_n\to \mathbb{R}^d$, then $\dim_L(X)\leq d$.
\end{corollary}

\subsection{Bi-Lipschitz Embeddings into $\ell^p$-sums}\label{s:banach}

In the spirit of the previous two subsections, we demonstrate how uniformly bi-Lipschitz embeddings of pieces of a geometric tree-like decomposition can be added up to form a bi-Lipschitz embedding of the entire space.

\begin{theorem} \label{thm:Banachembed}
Let $X = \{X_n\}_{n \in N}$ be a $C$-geometric tree-like decomposition of a metric space $X$. Let $L < \infty$ and $(\B_n,\|\cdot\|_n)_{n \in N}$ be a sequence of Banach spaces such that, for each $n \in N$, $X_n$ $L$-bi-Lipschitz embeds into $\B_n$. Then for every $p \in [1,\infty)$, the space $X$ admits a $CL$-bi-Lipschitz embedding into the $\ell^p$-sum $\bigoplus^p_{n \in N}\B_n$.
\end{theorem}

\begin{proof}
Let $p \in [1,\infty)$. Let $\|\cdot\|_p$ denote the norm on $\bigoplus^p_{n \in N}\B_n$. For each $n \in \N$, let $\phi_n: X_n \to \B_n$ be an $L$-bi-Lipschitz embedding, and let $\iota_n: \B_n \hookrightarrow \bigoplus^p_{n \in N}\B_n$ denote the canonical inclusion. By subtracting $\phi_n(p_n)$ from $\phi_n$, we may assume $\phi_n(p_n) = 0$, so that $\phi_n \in \Lip_0(X_n;\B_n)$. By Theorem \ref{T:general}, there exists a $CL$-Lipschitz map $\phi: X \to \bigoplus^p_{n \in N}\B_n$ such that $\phi \big|_{X_n} - \phi(p_n) = \iota_n \circ \phi_n$. It remains to check the co-Lipschitz constant of $\phi$.

Let $x,y \in X$. Let $(z_i)_{i \in I}$ be a decomposition path from $x$ to $y$ with $\{z_{i-1},z_{i}\} \subset X_{n_i}$. By Remark \ref{R:minimal}, we may assume $(z_i)_{i\in I}$ is minimal. Then we have
\begin{align*}\|\phi(x)-\phi(y)\|_p^p&=\left\|\sum_{1\leq i\in I}(\phi(z_{i-1})-\phi(z_i))\right\|_p^p=\sum_{1\leq i\in I}\|\phi_{n_i}(z_{i-1})-\phi_{n_i}(z_{i})\|_{n_i}^p \\
&\geq\max_{1\leq i\in I}\|\phi_{n_i}(z_{i-1})-\phi_{n_i}(z_i)\|_{n_i}^p\geq \max_{1\leq i\in I}(L^{-1}d(z_{i-1},z_i))^p\\
&\geq((CL)^{-1}d(x,y))^p,
\end{align*}
where the final inequality follows from the minimality of $(z_i)_{i\in I}$. Thus $\phi$ is $CL$-co-Lipschitz.
\end{proof}

The next corollary follows immediately from Theorem \ref{thm:Banachembed} and the facts that $\bigoplus^p_{n \in N} L^p([0,1]) = L^p([0,1])$ and $\bigoplus^p_{n \in N} \ell^p = \ell^p$ isometrically whenever $N$ is a countable indexing set.

\begin{corollary}\label{C:Banachembed}
Let $X = \{X_n\}_{n \in N}$ be a $C$-geometric tree-like decomposition of a metric space $X$. Let $L < \infty$ and $p \in [1,\infty)$. If $X_n$ $L$-bi-Lipschitz embeds into $L^p([0,1])$ (resp. $\ell^p$) for each $n$, then $X$ $CL$-bi-Lipschitz embeds into $L^p([0,1])$ (resp. $\ell^p$).
\end{corollary}

\section{Applications to QC Trees}\label{S:tree_apps}
In order to apply the results from previous sections to the case that $X$ is a quasiconformal tree $T$, we first establish relevant definitions and notation. Here we closely follow the notation of \cite{DEV21}. For the remainder of this section, fix a $1$-bounded turning $QC$ tree $T$.

We write $\mathcal{L}(T)$ to denote the leaves of $T$. Let $\{\mathcal{N}_n\}_{n\in\mathbb{N}}$ be a sequence such that, for each $n\in\mathbb{N}$, the set $\mathcal{N}_n$ is a $2^{-n}$-net in $\mathcal{L}(T)$ and $\mathcal{N}_n\subset \mathcal{N}_m$ for $n\leq m$. We then define 
\[T_n=\bigcup_{x,y\in \mathcal{N}_n}[x,y].\]
Here $[x,y]$ denotes the unique arc joining $x$ to $y$ in $T$. For $n\geq 2$, we write $\{K_n^j\}_{j\in J_n}$ to denote the collection of pairwise disjoint connected components of the compact set $\overline{T_n\setminus T_{n-1}}$. We also write $K_1^1=T_1$ and $J_1=\{1\}$. For notational convenience, we write the combination of collections $\{K_n^j\}_{j\in J_n}$ as $\{K_n^j\}_{(n,j)\in I}$. In this context,
\[I :=\{(n,j)\,|\,n\in \mathbb{N}, j\in J_n\}.\]

If $\mathcal{L}(T)$ is an infinite set, then we cannot assume that $T=\bigcup_{(n,j)\in I}K_n^j$. However, as pointed out in \cite[Remark 3.2]{DEV21}, it is true that
\begin{equation}\label{eq:L(T)^cdense}
    T\setminus\mathcal{L}(T)\subset \bigcup_{(n,j)\in I}K_n^j.
\end{equation}
Defining \[T_\infty:=\bigcup_{m\in\mathbb{N}}T_m=\bigcup_{(n,j)\in I}K_n^j,\] 
we note that \eqref{eq:L(T)^cdense} implies $T_{\infty} \subset T$ is dense.

\begin{definition}\label{D:DEBV}
Given a QC tree $T$, we call the collection of subtrees $\{K_n^j\}_{(n,j)\in I}$ constructed above a \textit{DEBV decomposition} of the corresponding subtree $T_\infty\subset T$.
\end{definition}

Our goal is to show that any DEBV decomposition of $T_\infty\subset T$ constitutes a geometric tree-like decomposition of $T_\infty$. By Remark \ref{rmk:finitetreedecomp}, it suffices to show that, for each $(n,j) \in I$, the DEBV decomposition $\{K_{n'}^{j'}\}_{(n',j') \leq (n,j)}$ of the finite-leaved truncation 
$\bigcup_{(n',j') \leq (n,j)} K_{n'}^{j'}$ is a $C$-geometric tree-like decomposition, where $C$ is some constant independent of $(n,j)$. Towards this end, we will assume from here till Theorem \ref{thm:arc_qa_decomp} that $T$ has finitely many leaves. We will also rescale the metric on $T$ so that $\diam(T) = 1$, which leaves the bounded turning constant unchanged. To summarize, $T$ is 1-bounded turning, and we will assume that it has finitely many leaves and that $\diam(T) = 1$. These assumptions match those in \cite[Section 3]{DEV21}, and thus we may directly cite results therein. We  begin with the following modified version of \cite[Lemma 3.3]{DEV21}.

\begin{lemma}\label{L:gammas}
Assume that $T$ has finitely many leaves and that $\diam(T) = 1$. Given a subtree $K_n^j$ in a fixed DEBV decomposition of $T$ and a point $p^j_{n,1} \in K_n^j$, the subtree $K_n^j$ consists of at most $C$ arcs $\{\gamma_{n,m}^j\}_{m\in M_n^j}$ (where $C$ depends only on the doubling constant of $T$) such that the indexing satisfies
\begin{itemize}
    \item $M_n^j = \{1,\dots \max(M_n^j)\}$.
    \item $p^j_{n,1} \in \gamma_{n,1}^j$,
    \item for $2\leq m\leq \max(M_n^j)$, there exists a unique point $p_{n,m}^j\in\gamma_{n,m}^j\cap\bigcup_{l<m}\gamma_{n,l}^j$.
\end{itemize}
\end{lemma}

\begin{proof}
Let $K_n^j$ be a subtree in a fixed DEBV decomposition of $T$ and $p^j_{n,1} \in K_n^j$. Since $n$ and $j$ are fixed for this proof, we write $K$ in place of $K_n^j$ and $p_1$ in place of $p^j_{n,1}$. By \cite[Lemma 3.3]{DEV21}, $K$ is the union of arcs $\{\tilde{\gamma}_{m}\}_{m \in M}$ with $M = \{1,\dots \max(M)\}$, where $\max(M) \leq C'$ and $C'$ depends only on the doubling constant of $T$. Without loss of generality, we may assume that $T$ is not covered by any proper subcollection. We recursively (in $m$) construct an indexed collection of arcs $\{\gamma_{(l,i)}\}_{(l,i)} \in M_{\leq m} \times \{1,2\}$ (with $M_{\leq m} \times \{1,2\}$ ordered lexicographically)  such that
\begin{enumerate}
    \item\label{item:gammas1} $p_1 \in \gamma_{(1,1)}$,
    \item\label{item:gammas2} for each $(1,1) < (l,i) \leq (m,2)$, there exists a unique point $p_{(l,i)}\in\gamma_{(l,i)}\cap\bigcup_{(l',i')<(l,i)}\gamma_{(l',i')}$, and
    \item\label{item:gammas3} there exists a bijection $\sigma: M_{\leq m} \to M_{\leq m}$ such that $\bigcup_{(l,i) \leq (m,2)} \gamma_{(l,i)}$ is a subtree of $K$ and $\bigcup_{(l,i) \leq (m,2)} \gamma_{(l,i)} = \bigcup_{l \leq m} \tilde{\gamma}_{\sigma(l)}$.
\end{enumerate}

For the base case, let $m_1 \in M$ such that $p_1 \in \tilde{\gamma}_{m_1}$. Then $\tilde{\gamma}_{m_1} = [x,p_1] \cup [p_1,y]$ for some $x,y \in \tilde{\gamma}_{m_1}$ with $[x,p_1) \cap (p_1,y] = \emptyset$. We set $\gamma_{(1,1)} := [x,p_1]$, $\gamma_{(1,2)} := [p_1,y]$, and note that \eqref{item:gammas1}-\eqref{item:gammas3} are trivially satisfied.

Assume that the arcs have been constructed for some $1 \leq m < \max(M)$. Since no proper subcollection of $\{\tilde{\gamma}_{l}\}_{l \in M}$ covers $K$, \eqref{item:gammas3} implies that \\ $\{\gamma_{(l,i)}\}_{(l,i) \leq (m,2)}$ does not cover $K$. Then there must exist an arc $\tilde{\gamma}_{l_{m+1}}$ such that $\tilde{\gamma}_{l_{m+1}} \cap \bigcup_{(l,i) \leq (m,2)} \gamma_{(l,i)} \neq \emptyset$ and $\tilde{\gamma}_{l_{m+1}} \not\subset \bigcup_{(l,i) \leq (m,2)} \gamma_{(l,i)}$. Since \\ $\bigcup_{(l,i) \leq (m,2)} \gamma_{(l,i)}$ is a subtree of the tree $K$, it follows that there exist nonempty closed subarcs $[w,x], [x,y], [y,z]$ of $\tilde{\gamma}_{l_{m+1}}$ such that
\begin{itemize}
    \item $\tilde{\gamma}_{l_{m+1}} = [w,x] \cup [x,y] \cup [y,z]$,
    \item $[w,x), (x,y), (y,z]$ are pairwise disjoint, and
    \item $[x,y] \subset \bigcup_{(l,i) \leq (m,2)} \gamma_{(l,i)}$.
\end{itemize}
Here we allow any of these closed intervals to degenerate into single points, and in such cases, two of the closed subarcs may coincide. We then set $\gamma_{(m+1,1)} := [w,x]$ and $\gamma_{(m+1,2)} := [y,z]$. This completes the recursive construction, and we leave the straightforward verification of \eqref{item:gammas1}-\eqref{item:gammas3} to the reader. The conclusion of the lemma is satisfied by choosing $M_{n}^j := M \times \{1,2\}$ ordered lexicographically, $\gamma_{n,m}^j := \gamma_m$, and $C := 2C'$.
\end{proof}

\begin{lemma}\label{L:arcs_qd}
Assume $T$ has finitely many leaves and $\diam(T) = 1$. If $\{K_n^j\}_{(n,j)\in I}$ is a DEBV decomposition of $T$, then, for any fixed $(n,j)\in I$, the collection of arcs $\{\gamma_{n,m}^j\}_{m\in M_n^j}$ given by Lemma~\ref{L:gammas} constitutes a $C$-geometric tree-like decomposition of $K_n^j$. Here, $C$ depends only on the doubling constant of $T$.
\end{lemma}

\begin{proof}
Since $n$ and $j$ are fixed for this proof, we write $K$ in place of $K_n^j$ and $\{\gamma_m\}_{m\in M}$ in place of $\{\gamma_{n,m}^j\}_{m\in M_n^j}$. By Lemma \ref{L:gammas}, $\{\gamma_{m}\}_{m\in M}$ is a tree-like decomposition of $K$.

Let $x,y\in K$. Since $K$ is a compact tree, for each $m\in M$, the intersection $\gamma_m\cap[x,y]$ is a compact subarc $[a_m,b_m]$. By re-indexing the points $\{a_m,b_m\}_{m \in M}$, we obtain a sequence $(z_l)_{l\in M'}\subset K$ such that
\begin{enumerate}
    \item{$z_0=x$ and $z_{\max(M')}=y$,}
    \item{for each $1\leq l\leq\max(M')$, there exists $m_l$ such that  $\{z_{l-1},z_l\}\subset\gamma_{m_l}$, and}
    \item{\[C^{-1}\sum_{1\leq l\leq\max(M')}d(z_{l-1},z_l)\leq d(x,y)
    .\]}
\end{enumerate}
The final item is due to the facts that $\max(M') \leq \max(M) \leq C$, where $C$ is the constant in the statement of Lemma \ref{L:gammas}, and the assumption that $T$ is $1$-bounded turning. Thus we find that (1) and (3) of Definition \ref{D:decomp} hold. To verify (2), let $(z_i)_{i\in I}$ denote any minimal decomposition path from $x$ to $y$ in $K$. By Lemma \ref{L:gammas}, there are no more than $C-1$ branch points in $K$. Since $(z_i)_{i\in I}$ consists of pairwise distinct branch points (except possibly $z_0=x$ and $z_{\max(I)}=y$), we have
\begin{enumerate}
    \item[(4)]{\[d(x,y)\leq \sum_{1\leq i\leq \max(I)}d(z_{i-1},z_i)\leq C\max_{1\leq i\leq \max(I)}d(z_{i-1},z_i).\]}
\end{enumerate}
\end{proof}

Given $T$ (under the current assumption that $T$ has finitely many leaves), suppose $\{S_n\}_{n\in N}$ is a tree-like decomposition of $T$ such that each $S_n$ is a sub-tree of $T$. Let $\gamma$ denote any non-degenerate compact arc in $T$. Let $x,y$ be the endpoints of $\gamma$, and equip $\gamma$ with an ordering coming from a homeomorphism to $[0,1]$ such that $x<y$. Write $(S_{n_j})_{j\in J'}$ to denote the (finite) sequence of sub-trees $S_{n_j}$ traversed by $\gamma$ in consecutive order along $\gamma$, where $J':=\{1,\dots,\max(J')\}$. By \textit{traversed}, we mean that $\gamma\cap S_{n_j}$ contains more than one point. Thus, by \cite[Theorem 10.10]{Nadler92}, each intersection $S_{n_j}\cap\gamma$ (being the intersection of two connected sets in $T$) is equal to a single non-degenerate subarc $\gamma_j\subset \gamma$. By \textit{consecutive order along $\gamma$}, we mean that $j < j'$ whenever $j,j' \in J'$ and there exist points $p \in \gamma\cap S_{n_j}$ and $p' \in \gamma\cap S_{n_{j'}}$ with $p < p'$. Furthermore, for $j<\max(J')$, this construction implies that $S_{n_j}\not=S_{n_{j+1}}$. Define $J:=\{0,\dots,\max(J')\}=\{0\}\cup J'$. For $1\leq j\leq\max(J)$, write $\{w_{j-1},w_{j}\}\subset S_{n_j}$ to denote the points of $\gamma$ such that $\gamma_j=[w_{j-1},w_j]$.

It is easy to see that the sequence $(w_j)_{j\in J}$ is a decomposition path in $T$ with respect to $\{S_n\}_{n\in N}$. We refer to $(w_j)_{j\in J}$ as the decomposition path in $\gamma$ \textit{induced by} $\{S_n\}_{n\in N}$, and to the arcs $\{\gamma_j\}_{j\in J'}$ as the subarcs of $\gamma$ \textit{induced by} $\{S_n\}_{n\in N}$.

\begin{lemma}\label{L:path_to_arc}
Suppose $T$ has finitely many leaves and $\{S_n\}_{n\in N}$ is a tree-like decomposition of $T$ such that each $S_n$ is a subtree of $T$. 
\begin{enumerate}
    \item{If $(z_i)_{i\in I}$ is a minimal decomposition path with respect to $\{S_n\}_{n\in N}$, then $\gamma:=\bigcup_{1\leq i\leq \max(I)}[z_{i-1},z_i]$ is an arc, and $\{[z_{i-1},z_i]\}_{1\leq i\leq \max(I)}$ is the collection of subarcs in $\gamma$ induced by $\{S_n\}_{n\in N}$.}
    \item{Conversely, given any compact arc $\gamma\subset T$, if $(w_j)_{j\in J}$ is the decomposition path in $\gamma$ induced by $\{S_n\}_{n\in N}$, then $(w_j)_{j\in J}$ is minimal.}
\end{enumerate}
\end{lemma}

\begin{proof}
Let $(z_i)_{i\in I}$ denote a minimal decomposition path with respect to $\{S_n\}_{n\in N}$. By definition of a decomposition path, for $1\leq i\leq\max(I)$, there exists $n_i\in N$ such that $\{z_{i-1},z_i\}\subset S_{n_i}$. Define $\gamma:=\bigcup_{1\leq i\leq \max(I)}\gamma_i$, where, for $1\leq i\leq \max(I)$, we define $\gamma_i:=[z_{i-1},z_i]$. Since each $S_{n_i}$ is a subtree, we have $\gamma_i\subset S_{n_i}$. 

By way of contradiction, assume $\gamma$ is not an arc. Since $\gamma$ is the union of arcs $\{\gamma_i\}_{i\in I}$, this assumption means that there exists some minimal index $2\leq i_1\leq \max(I)$ such that $\bigcup_{1\leq i<i_1}\gamma_i=[z_0,z_{i_1-1}]$ is an arc and
\[[z_0,z_{i-1}] \cap (z_{i_1-1},z_{i_1}] \not=\emptyset.\]
In particular, $\gamma_{i_1}$ intersects $[z_0,z_{i-1}]$ in at least two points: the point $z_{i_1-1}$ and some other point we denote by $z_{i_1}'$. Since $T$ is a tree, by \cite[Theorem 10.10]{Nadler92} we must have $[z_{i_1}',z_{i_1-1}]\subset \gamma_{i_1}\cap[z_0,z_{i_1-1}]$. In particular, either $[z_{i_1}',z_{i_1-1}]\subset [z_{i_1-2},z_{i_1-1}]$ or $[z_{i_1-2},z_{i_1-1}]\subset [z_{i_1}',z_{i_1-1}]$. In either case, $\gamma_{i_1}\cap \gamma_{i_1-1}$ is non-trivial, and so $S_{n_{i_1}}\cap S_{n_{i_1-1}}$ is non-trivial. Since $\{S_n\}_{n\in N}$ is a tree-like decomposition of $T$, we must have $S_{n_{i_1}}=S_{n_{i_1-1}}$. But this contradicts the fact that $(z_i)_{i\in I}$ is minimal. Therefore, $\gamma$ is an arc.

We claim that, for each $i\in I$, we have $\gamma_i=\gamma\cap S_{n_i}$. By way of contradiction, assume the contrary: there exists $i_0\in I$ such that $\gamma_{i_0}\subsetneq \gamma\cap S_{n_0}$. Since $S_{n_0}$ is a tree and $\gamma$ is an arc, it must be true that $\gamma\cap S_{n_0}$ is an arc (via \cite[Theorem 10.10]{Nadler92}). Denote this arc by $\gamma_{i_0}'$. Since $\gamma_{i_0}'\cap (\gamma\setminus\gamma_{i_0})\not=\emptyset$, we have $\gamma_{i_0}'\cap(z_{i_2-1},z_{i_2})\not=\emptyset$ for some $i_2\not=i_0$. But this implies that $S_{n_{i_0}}\cap S_{n_{i_2}}$ is non-trivial, from which follows a contradiction to the minimality of $(z_i)_{i\in I}$. Therefore, for every $i\in I$, we have $\gamma_i=\gamma\cap S_{n_i}$. In other words, $\{\gamma_i\}_{1\leq i\leq \max(I)}$ is the collection of subarcs in $\gamma$ induced by $\{S_n\}_{n\in N}$.

To prove the converse statement of the lemma, let $\gamma$ denote any compact arc in $T$, and let $(w_j)_{j\in J}$ denote the decomposition path in $\gamma$ induced by $\{S_n\}_{n\in N}$. Suppose, by way of contradiction, that $(w_j)_{j\in J}$ is not minimal. By Remark \ref{R:minimal}, there exist indices $1\leq j_0<j_1\in J$ such that $w_{j_1}\in S_{n_{j_0}}$. Since $S_{n_{j_0}}$ is a tree, we must have $[w_{j_0-1},w_{j_1}]\subset S_{n_{j_0}}$. In particular, $[w_{j_0},w_{j_0+1}]\subset S_{n_{j_0}}$. But then $S_{n_{j_0+1}}=S_{n_{j_0}}$, which contradicts the construction of the sequence $(S_{n_j})_{j\in J}$. Therefore, $(w_j)_{j\in J}$ is minimal. 
\end{proof}

\begin{proposition}\label{P:arc_decomp}
Assume that $T$ has finitely many leaves and that $\diam(T)$ $= 1$. Let $\{K_n^j\}_{(n,j)\in I}$ be a DEBV decomposition of T. Then there exists $C_1 < \infty$, depending only on the doubling constant of $T$, with the following property: Given any compact arc $\gamma \subset T$, the subarcs $\{\gamma_i\}_{i\in I'}\subset \gamma$ induced by $\{K_n^j\}_{(n,j)\in I}$ satisfy
\[C_1^{-1}\sum_{i\in I'} \diam(\gamma_i) \leq \diam(\gamma) \leq C_1 \max_{i\in I'} \diam(\gamma_i).\]
\end{proposition}

\begin{proof}
The proof is essentially contained in \cite{DEV21}, which we follow. Let $\gamma \subset T$ be a compact arc, and let $\{\gamma_i\}_{i\in I'}$ denote the subarcs of $\gamma$ induced by $\{K_n^j\}_{(n,j)\in I}$, where $I'=\{1,\dots,\max(I')\}$. By \cite[Lemma 3.3]{DEV21}, we have $\diam(\gamma_i) \leq \min\{\diam(\gamma),2^{2-m_i}\}$, where $m_i$ is such that $\gamma_i\subset K_{m_i}^{j_i}$. By \cite[Lemma 3.6]{DEV21}, there exists $i_0 \in I'$ such that $m_{i+1} < m_i$ for all $i < i_0$ and $m_{i+1} > m_i$ for all $i \geq i_0$. Choose $n \in \N$ such that $2^{-n-1} \leq \diam(\gamma) \leq 2^{-n}$, and let $i_* := \min\{i: m_i \leq n\}$ and $i^* := \max\{i: m_i \leq n\}$ (or $i_* = i^* = i_0$ if $\{i: m_i\leq n\}=\emptyset$). By \cite[Lemmas 3.5 and 3.6]{DEV21}, there exists $M < \infty$ (depending only on the doubling constant of $T$) such that $i^* - i_* \leq M$. Now we prove the first inequality of the second item in the lemma:
\begin{align*}
    \sum_{i\in I'} \diam(\gamma_i) = \sum_{i < i_*} \diam(\gamma_i) + \sum_{i_*\leq i\leq i^*} \diam(\gamma_i) + \sum_{i > i^*} \diam(\gamma_i) \\
    \leq \sum_{i < i_*} 2^{2-m_i} + \sum_{i_*\leq i\leq i^*} \diam(\gamma) + \sum_{i > i^*} 2^{2-m_i} \\
    \leq 2\sum_{m > n} 2^{2-m} + M\diam(\gamma) \\
    = 2^{3-n} + M \diam(\gamma) \\
    \leq (16+M)\diam(\gamma).
\end{align*}
We conclude by proving the second inequality. Choose $N \in \N$ large enough so that $2\sum_{m > n+N} 2^{2-m} < 2^{-2-n}$. Obviously, such an $N$ exists and can be chosen independently of $n$. Then we have
\begin{align*}
    2^{-1-n} \leq \diam(\gamma) \leq \sum_{i\in I'} \diam(\gamma_i) \\
    = \sum_{i < i_*-N} \diam(\gamma_i) + \sum_{i_*-N\leq i\leq i^*+N} \diam(\gamma_i) + \sum_{i > i^*+N} \diam(\gamma_i) \\
    \leq \sum_{i < i_*-N} 2^{2-m_i} + \sum_{i_*-N\leq i\leq i^*+N} \max_{i\in I} \diam(\gamma_i) + \sum_{i > i^*+N} 2^{2-m_i} \\
    \leq 2 \sum_{m > n+N} 2^{2-m} + (M+2N) \max_{i\in I'} \diam(\gamma_i) \\
    < 2^{-2-n} + (M+2N) \max_{i\in I'} \diam(\gamma_i).
\end{align*}
Solving this inequality for $\max_{i\in I'} \diam(\gamma_i)$ yields
\[\max_{i\in I'} \diam(\gamma_i) \geq \frac{1}{M+2N} 2^{-2-n} \geq \frac{2^{-2}}{M+2N}\diam(\gamma).\]
\end{proof}

\begin{lemma}\label{L:trees_qd}
Assume that $T$ has finitely many leaves and $\diam(T) = 1$. If $\{K_n^j\}_{(n,j)\in I}$ is a DEBV decomposition of $T$, then $\{K_n^j\}_{(n,j)\in I}$ is a $C_1$-geometric tree-like decomposition of $T$, with $C_1$ as in Proposition \ref{P:arc_decomp}.
\end{lemma}

\begin{proof}
We first apply \cite[Lemma 3.1(ii),(iii)]{DEV21} to conclude that, for each $(1,1)<(n,j)\in I$, there exists a unique point $p_n^j\in K_n^j\cap \bigcup_{(m,i)<(n,j)}K_m^i$. Thus $\{K_n^j\}_{(n,j)\in I}$ constitutes a tree-like decomposition of $T$, and we verify property (1) in Definition \ref{D:decomp}.

To verify (2) in Definition \ref{D:decomp}, let $(z_i)_{i\in J}$ denote a minimal decomposition path with respect to the tree-like decomposition $\{K_n^j\}_{(n,j)\in I}$ joining any two points $x$ and $y$ in $T$. Furthermore, assume that, for $1\leq i\leq\max(J)$, we have $\{z_{i-1},z_i\}\subset K_{n_i}^{j_i}$. Define $\gamma:=\bigcup_{1\leq i\leq \max(J)}\gamma_i$, where, for $1\leq i\leq \max(J)$, we define $\gamma_i:=[z_{i-1},z_i]\subset K_{n_i}^{j_i}$. By Lemma \ref{L:path_to_arc}, $\gamma$ is an arc, and $\{\gamma_i\}_{1\leq i\leq \max(J)}$ is the collection of subarcs in $\gamma$ induced by $\{K_n^j\}_{(n,j)\in I}$. By Proposition \ref{P:arc_decomp} and $1$-bounded turning property of $T$, we verify (2) of Definition \ref{D:decomp}. 

To verify (3) in Definition \ref{D:decomp}, let $x,y\in T$. Write $(w_j)_{j\in J}$ to denote the decomposition path in $[x,y]$ and $\{\gamma_j\}_{1\leq j\leq \max(J)}=\{[w_{j-1},w_j]\}_{1\leq j\leq\max(J)}$ to denote the subarcs of $[x,y]$ induced by $\{K_n^j\}_{(n,j)\in I}$. By Proposition \ref{P:arc_decomp} and the $1$-bounded turning property of $T$, we verify (3) of Definition \ref{D:decomp}.

Having confirmed properties (1)-(3) of Definition \ref{D:decomp}, we thus confirm that $\{K_n^j\}_{(n,j)\in I}$ is a $C_1$-geometric tree-like decomposition of $T$.
\end{proof}



From here, we no longer assume that $T$ has finitely many leaves. Given a DEBV decomposition $\{K_n^j\}_{(n,j)\in I}$ of $T_\infty\subset T$, we note that, for any $(n,j)\in I$, the tree $\bigcup_{(n',j')\leq(n,j)}K_{n'}^{j'}$ has finitely many leaves, and $\{K_{n'}^{j'}\}_{(n',j')\leq(n,j)}$ constitutes a DEBV decomposition for this tree. Therefore, by Lemma \ref{L:trees_qd}, for every $(n,j)\in I$, there exists a unique point $p^j_{n,1}\in K_n^j\cap \bigcup_{(n',j')<(n,j)}K_{n'}^{j'}$. We also write $\{\gamma_{n,m}^j\}_{m\in M_n^j}$ to denote the collection of subarcs of $K_n^j$ given to us by Lemma \ref{L:gammas}. Write $K:=\{(n,j,m)\,|\,(n,j)\in I, m\in M_n^j\}$, and endow $K$ with the lexicographic ordering. We use this terminology to state the following lemma.

\begin{theorem}\label{thm:arc_qa_decomp}
If $T$ is a $1$-bounded turning QC tree and $\{K_n^j\}_{(n,j)\in I}$ is a DEBV decomposition of $T_\infty$, then the arcs $\{\gamma_{n,m}^j\}_{(n,j,m)\in K}$ constitute a $C_1C$-geometric tree-like decomposition of $T_\infty$, with $C$ as in Lemma \ref{L:arcs_qd} and $C_1$ as in Proposition \ref{P:arc_decomp}. 
\end{theorem}

\begin{proof}
As mentioned after the definition of DEBV decompositions, Remark \ref{rmk:finitetreedecomp} allows us to assume that $T$ has finitely many leaves. It is also clear that a $C_1C$-geometric tree-like decomposition of $T$ remains a $C_1C$-geometric tree-like decomposition after rescaling the metric, so we may assume that $\diam(T) = 1$. Hence, the standing assumptions of this section are now met, and we are in position to use the results herein.

We first confirm that $\{\gamma_{n,m}^j\}_{(n,j,m)\in K}$ forms a tree-like decomposition of $T_\infty$. Let $(n,j,m) 
\in K$ with $(1,1,1) < (n,j,m)$. First assume that $m = 1$. Then
\[\bigcup_{(n',j',m') < (n,j,1)} \gamma^{j'}_{n',m'} = \bigcup_{(n',j') < (n,j)} K^{j'}_{n'}\]
and $\gamma_{n,1}^j \subset K_{n}^j$, and thus $\gamma_{n,1}^j \cap \bigcup_{(n',j',m') < (n,j,1)} \gamma^{j'}_{n',m'}$ consists of at most one point by Lemma \ref{L:trees_qd}. By Lemma \ref{L:gammas}, $p_{n,1}^j \in \gamma_{n,1}^j$, and thus $p_{n,1}^j \in \gamma_{n,1}^j \cap \bigcup_{(n',j',m') < (n,j,1)} \gamma^{j'}_{n',m'}$. Now assume $m > 1$. Then $\bigcup_{(n',j',m') < (n,j,m)} \gamma^{j'}_{n',m'} = \bigcup_{(n',j') < (n,j)} K^{j'}_{n'} \cup \bigcup_{m'<m} \gamma^{j}_{n,m'}$ and $\gamma_{n,m}^j \subset K_{n}^j$. This gives us
\begin{align*}
    \gamma_{n,m}^j \cap &\bigcup_{(n',j',m') < (n,j,m)} \gamma^{j'}_{n',m'} \\
    &\subset \left(\gamma_{n,m}^j \cap \left(K_n^j \cap \bigcup_{(n',j') < (n,j)} K^{j'}_{n'}\right)\right) \cup \left(\gamma_{n,m}^j \cap \bigcup_{m'<m} \gamma^{j}_{n,m'}\right) \\
    &= \left(\gamma_{n,m}^j \cap \{p_{n,1}^j\}\right) \cup \left(\gamma_{n,m}^j \cap \bigcup_{m'<m} \gamma^{j}_{n,m'}\right) \\
    &= \gamma_{n,m}^j \cap \bigcup_{m'<m} \{p_{n,1}^j\} \cup \gamma^{j}_{n,m'} \\
    &= \gamma_{n,m}^j \cap \bigcup_{m'<m} \gamma^{j}_{n,m'} \\
    &= \{p_{n,m}^j\},
\end{align*}
showing that $\gamma_{n,m}^j \cap \bigcup_{(n',j',m') < (n,j,m)} \gamma^{j'}_{n',m'}$ consists of at most one point. The other containment $p_{n,m}^j \in \gamma_{n,m}^j \cap \bigcup_{(n',j',m') < (n,j,m)} \gamma^{j'}_{n',m'}$ is easy to see via Lemma \ref{L:gammas}, which completes the proof that $\{\gamma_{n,m}^j\}_{(n,j,m)\in K}$ is a tree-like decomposition of $T_\infty$.

To show that $\{\gamma_{n,m}^j\}_{(n,j,m)\in K}$ is a geometric tree-like decomposition of $T_\infty$, let $(z_l)_{l\in L}$ denote a minimal decomposition path (with respect to the tree-like decomposition $\{\gamma_{n,m}^j\}_{(n,j,m)\in K}$) in $T_\infty$. By Lemma \ref{L:path_to_arc}, the union $\bigcup_{1\leq l\leq \max(L)}[z_{l-1},z_l]$ is an arc, which we denote by $\gamma$. Furthermore, for each $1\leq l\leq \max(L)$, there exists a unique $(n_l,j_l,m_l)\in K$ such that $\gamma\cap \gamma_{n_l,m_l}^{j_l}=[z_{l-1},z_l]$. Writing $(K_{n_s}^{j_s})_{s\in S}$ to denote the sequence of sets $K_n^j$ traversed by $\gamma$ in consecutive order along $\gamma$, then, for each $1\leq s\leq \max(S)$, we note that the subarc $\gamma\cap K_{n_s}^{j_s}$ satisfies
\begin{align*}
\gamma\cap K_{n_s}^{j_s}&=\bigcup_{m\in M_{n_s}^{j_s}}\gamma\cap\gamma_{n_s,m}^{j_s}\\
&=\bigcup\{[z_{l-1},z_l]\,|\,l\in L \text{ such that } (n_l,j_l,m_l)=(n_s,j_s,m_l)\}.
\end{align*}
We thus obtain a sub-sequence $(z_{l_s})_{s\in S}\subset(z_l)_{l\in L}$ such that, for each $1\leq s\leq\max(S)$, we have 
\[\gamma\cap K_{n_s}^{j_s}=[z_{l_{s-1}},z_{l_s}]=\bigcup_{l_{s-1}< l\leq l_s}[z_{l-1},z_l].\]
In particular, we have
\[\{z_l\,|\,l_{s-1}\leq l\leq l_{s}\}\subset[z_{l_{s-1}},z_{l_s}]=\gamma\cap K_{n_s}^{j_s}.\] 
We also note that $z_{l_0}=z_0$ and $z_{l_{\max(S)}}=z_{\max(L)}$. Therefore, $(z_{l_s})_{s\in S}$ is the decomposition path in $\gamma$ induced by $\{K_n^j\}_{(n,j)\in I}$, and each sequence $(z_l)_{l_s\leq l\leq l_{s+1}}$ is a decomposition path in $K_{n_s}^{j_s}$ with respect to the tree-like decomposition $\{\gamma_{n_s,m}^{j_s}\}_{m\in M_{n_s}^{j_s}}$. Since $(z_l)_{l\in L}$ is minimal, each $(z_l)_{l_s\leq l\leq l_{s+1}}$ is minimal. By Lemma \ref{L:path_to_arc}, the decomposition path $(z_{l_s})_{s\in S}$ is also minimal. By Lemmas \ref{L:trees_qd} and \ref{L:arcs_qd}, we have
\begin{align*}
d(z_0,z_{\max(L)})&\leq C_1\max_{s<\max(S)}d(z_{l_s},z_{l_{s+1}})\\
&\leq C_1C\max_{s<\max(S)}\max_{l_s\leq l< l_{s+1}}d(z_l,z_{l+1})=C_1C\max_{l<\max(L)}d(z_l,z_{l+1}).
\end{align*}
Furthermore, by Proposition \ref{P:arc_decomp} and the $1$-bounded turning property of $T$, the path $(z_{l_s})_{s\in S}$ is a $C_1$-short decomposition path with respect to $\{K_n^j\}_{(n,j)\in I}$. By Lemma \ref{L:gammas} and the $1$-bounded turning property of $T$, each path $(z_l)_{l_s\leq l\leq s_{s+1}}$ is a $C$-short decomposition path with respect to $\{\gamma_{n_s,m}^{j_s}\}_{m\in M_{n_s}^{j_s}}$. Therefore
\begin{align*}
\sum_{l<\max(L)}d(z_l,z_{l+1})&=\sum_{s<\max(S)}\sum_{l_s\leq l<l_{s+1}}d(z_l,z_{l+1})\\
&\leq C\sum_{s<\max(S)}d(z_{l_s},z_{l_{s+1}})\leq CC_1\,d(z_0,z_{\max(L)}).
\end{align*}
In conclusion, $\{\gamma_{n,m}^j\}_{(n,j,m)\in K}$ constitutes a $C_1C$-geometric tree-like decomposition of $T_\infty$.
\end{proof}

\subsection{The Lipschitz Free Space of a QC Tree} Here we apply Theorem \ref{T:general} to the case that $X$ is a quasiconformal tree, thus proving Theorem \ref{T:main_trees}.

\begin{corollary}\label{C:tree_arcs}
Suppose $T$ is a QC tree, and $\{K_n^j\}_{(n,j)\in I}$ is a DEBV decomposition of $T_\infty$. then $\Lip_0(T) = \Lip_0(T_\infty)$ is linearly weak*-isomorphic to $\bigoplus_{(n,j,m)\in K}^\infty\Lip_0(\gamma_{n,m}^j)$. Consequently, $\F(T)$ is isomorphic to $L^1(Z)$ for some measure space $Z$, where $Z$ is purely atomic if and only if $T$ is purely 1-unrectifiable.
\end{corollary}

\begin{proof}
Let $B$ be the bounded turning constant of $T$ and $D$ its doubling constant. Then by \cite[Lemma~2.5]{BM20a}, $T$ is $B$-bi-Lipschitz equivalent to a 1-bounded turning, $D'$-doubling tree, where $D'$ only depends on $B,D$. Since $B$-bi-Lipschitz equivalent metric spaces have $B$-weak* isomorphic spaces of Lipschitz functions, we may assume that $T$ is 1-bounded turning and $D'$-doubling.

Since $T_\infty \subset T$ is dense, $\Lip_0(T) = \Lip_0(T_\infty)$. By Theorem \ref{thm:arc_qa_decomp}, \\ $\{\gamma_{n,m}^j\}_{(n,j,m)\in K}$ is a geometric tree-like decomposition of $T_\infty$. Therefore, by Theorem \ref{T:general}, we have that $\Lip_0(T)$ is linearly weak*-isomorphic to $\bigoplus_{(n,j,m)\in K}^\infty\Lip_0(\gamma_{n,m}^j)$. Hence, $\F(T)$ is isomorphic to $\bigoplus_{(n,j,m)\in K}^1\F(\gamma_{n,m}^j)$. Since each arc $\gamma_{n,m}^j$ is $1$-bounded turning and $D'$-doubling, Corollary \ref{cor:isoL1} implies that $\F(T)$ is isomorphic to $L^1(Z)$ for some measure space $Z$. As in the proof of Corollary \ref{cor:isoL1}, $Z$ is purely atomic $\Leftrightarrow$ $L^1(Z)$ has the Schur property $\Leftrightarrow$ $\F(T)$ has the Schur property $\Leftrightarrow$ $T$ is purely 1-unrectifiable.
\end{proof}

\subsection{The Lipschitz Dimension of a QC Tree}
We begin by citing the following result. 

\begin{theorem}[Theorem~2.2, \cite{Freeman20}]\label{T:arc_lip_dim}
Given a $1$-bounded turning Jordan arc $\gamma$, there exists an $L$-Lipschitz and $Q$-light map $f:\gamma\to\mathbb{R}$, where $L$ and $Q$ are absolute constants. 
\end{theorem}

\begin{remark}
To see that $L$ and $Q$ are absolute constants, by Theorem \ref{T:HM}, any $1$-bounded turning Jordan arc $\gamma$ with $\diam(\gamma)=1$ is $8$-bi-Lipschitz to a Jordan arc $\Gamma$ in $\mathcal{S}_1'$. Then, \cite[Lemmas 5.5 and 5.11]{Freeman20} demonstrate the existence of a $1$-Lipschitz $129$-light map $F:\Gamma\to\mathbb{S}$, where $\mathbb{S}$ denotes the unit circle equipped with a normalized arc-length distance. Since it is easy to see that the unit circle admits a $1$-Lipschitz $4$-light map into $\mathbb{R}$, we conclude (via the composition of these Lipschitz light maps) that $\Gamma$ admits a $1$-Lipschitz $Q'$-light map into $\mathbb{R}$, where $Q'$ is absolute (in fact, one can verify that $Q'=129\cdot4$). In the same way, it then follows that $\gamma$ admits an $8$-Lipschitz $Q$-light map into $\mathbb{R}$, where $Q$ is absolute. Furthermore, the assumption that $\diam(\gamma)=1$ is harmless in this context. Indeed, if we must scale $\gamma$ by $1/\diam(\gamma)$ in order to satisfy this assumption, then we simply rescale the resulting image in $\mathbb{R}$ by $\diam(\gamma)$. It is straightforward to verify that this will yield an $8$-Lipschitz $Q$-light map from $\gamma$ into $\mathbb{R}$. 
\end{remark}

While the proof of the following lemma is routine, it is included for the convenience of the reader.

\begin{lemma}\label{L:dense_ll}
Suppose $f:X\to\mathbb{R}$ is $L$-Lipschitz and $Q$-light. Then $f$ extends to an $L$-Lipschitz and $Q$-light map $\overline{f}:\overline{X}\to\mathbb{R}$, where $\overline{X}$ represents the completion of $X$. 
\end{lemma}

\begin{proof}
By continuity, $\overline{f}$ remains $L$-Lipschitz. To see that $\overline{f}$ is $Q$-light, let $E\subset \mathbb{R}$ be given such that $\delta:=\diam(E)>0$. Let $U\subset \overline{X}$ denote a $\delta$-component of $\overline{f}^{-1}(E)$ in $\overline{X}$. If $U\subset X$, then $U\subset f^{-1}(E)$ and, by assumption, $\diam(U)\leq Q\delta$. Suppose there exists a $\delta$-chain $(z_i)_{i\in I}$ in $U$ such that $(z_i)_{i\in I}\not\subset X$. Since $X$ is dense in $\overline{X}$, for any $\delta'>\delta$, there exists a $\delta'$-chain $(z'_i)_{i\in I}\subset X$ such that, for each $i\in I$, we have $d(z_i,z_i')<\varepsilon:=\delta'-\delta$. Since $f$ is $L$-Lipschitz, we also note that $f(\{z_i'\}_{i\in I})\subset E'$, where $E'$ is defined to be the $L\varepsilon$-neighborhood of $E$ in $\mathbb{R}$. Since $f$ is $Q$-light, we have 
\[\diam(\{z_i\}_{i\in I})\leq\diam(\{z_i'\}_{i\in I})+2\varepsilon\leq Q\diam(E')+2\varepsilon=Q\delta+2LQ\varepsilon+2\varepsilon.\] 
Since $\varepsilon>0$ is a arbitrarily small, it follows that $\diam(U)\leq Q\delta$. 
\end{proof}

\begin{theorem}\label{T:tree_dim1}
If $T$ is a $B$-bounded turning and $D$-doubling QC tree, then there exists an $L$-Lipschitz and $Q$-Light map $f:T\to\mathbb{R}$, where $L$ and $Q$ depend only on $B$ and $D$.
\end{theorem}

\begin{proof}
Since every $B$-bounded turning tree is $B$-bi-Lipschitz equivalent to a $1$-bounded turning tree (by \cite[Lemma~2.5]{BM20a}), it suffices to assume that $T$ is $1$-bounded turning and $D'$-doubling (where $D'$ depends only on $B,D$). Let $\{K_n^j\}_{(n,j)\in I}$ denote a DEBV decomposition of $T_\infty$. By Theorem \ref{thm:arc_qa_decomp}, the  collection of arcs $\{\gamma_{n,m}^j\}_{(n,j,m)\in K}$ associated with $\{K_n^j\}_{(n,j)\in I}$ is a $C$-geometric tree-like decomposition of $T_\infty$, where $C$ depends only on $D'$. By Theorem \ref{T:arc_lip_dim}, for each $(n,j,m)\in K$, there exists an $L$-Lipschitz and $Q$-light map $f_{(n,j,m)}:\gamma_{n,m}^j\to\mathbb{R}$. By Theorem \ref{T:general_LL}, there exists an $LC$-Lipschitz and $Q'$-light map $f:T_\infty\to \mathbb{R}$, where $Q'$ depends only on $L$, $Q$, and $C$. Since $T_\infty$ is dense in $T$, Lemma \ref{L:dense_ll} enables us to conclude that these properties are retained by the extension $\overline{f}:T\to\mathbb{R}$. 
\end{proof}

\section{Open Questions}

While the assumption of the doubling condition on $X$ is an essential ingredient to our proof of Proposition \ref{P:arc_decomp}, in light of Theorem \ref{T:arc_lip_dim}, it is not immediately obvious that the doubling property is a necessary assumption in Theorem \ref{T:main_dim}. Thus we ask:

\begin{question}
Is the Lipschitz dimension of a bounded turning (but not necessarily doubling) metric tree equal to 1?
\end{question}

Furthermore, while the assumption of the doubling property is necessary for our proof of Lemma \ref{L:dyadic}, we do not know if alternate methods could be used to identify the Lipschitz free space of a bounded turning Jordan arc. In particular, we ask:

\begin{question}
Is the Lipschitz free space of a (possibly non-doubling) bounded turning Jordan arc isomorphic to $L^1(Z)$ for some measure space $Z$?
\end{question}
\noindent This question generalizes one brought forth in the last paragraph of \cite[Section 8.6]{Weaver18}.

\appendix
\section{Quotients of $L^1$}
For the sake of self-containment, we include proofs of some well-known facts on isomorphisms between subspaces and quotient spaces of $L^1$. The arguments prioritize simplicity over obtaining the best possible isomorphism constants.

\begin{lemma} \label{lem:kernel}
Let $\mu$ be a finite measure on a measurable space $Y$, and let $K \subset L^1(\mu)$ denote the closed subspace consisting of all $g \in L^1(\mu)$ with $\int g \,d\mu = 0$. Then there exists a measurable subset $X \subset Y$ such that $K$ is 64-isomorphic to $L^1(\mu \llcorner X)$.
\end{lemma}

\begin{proof}
If $\mu = 0$, then $K = Y$ and $X := Y$ satisfies the desired property. Assume $\mu \neq 0$. First assume that $Y$ contains an atom $A \subset Y$. Then we choose $X := Y \setminus A$, and we get a 2-isomorphism $K \to L^1(\mu \llcorner X)$ given by $g \mapsto g\res_X$ with norm-2 inverse $h \mapsto h1_{X} - \left(\int_X h \, d\mu\right) \frac{1}{\mu(A)}1_A$.

Now assume that $Y$ contains no atoms. Let $\{A_n\}_{n\in\N}$ be a measurable partition of $Y$ with $\mu(A_n) = \frac{1}{2^{n+1}}\mu(Y)$ for every $n\in\N$. Let $K_0 \subset L^1(\mu)$ denote the closed subspace consisting of all $g \in L^1(\mu)$ with $\int_{A_0} g \,d\mu = 0$. Then the map $K \to K_0$ given by $g \mapsto g1_{Y \setminus A_0} + \left(g-\fint_{A_0}g \, d\mu\right)1_{A_0}$ is a norm-2 isomorphism with norm-2 inverse $h \mapsto h1_{Y \setminus A_0} + \left(h-\fint_{Y \setminus A_0}h \, d\mu\right)1_{A_0}$. Hence, it suffices to prove that $K_0$ is 16-isomorphic to $L^1(\mu \llcorner X)$ for some $X \subset Y$ measurable. We simply choose $X := Y$ and note that the map $K_0 \to L^1(\mu)$ given by
$$g \mapsto \sum_{n\in\N} \left(g - \fint_{A_n} g \, d\mu + \fint_{A_{n+1}} g \, d\mu \right)1_{A_n}$$
has norm 4 with norm-4 inverse
$$h \mapsto \left(h - \fint_{A_0} h \, d\mu\right)1_{A_0} + \sum_{n\geq 1} \left(h - \fint_{A_n} h \, d\mu + \fint_{A_{n-1}} h \, d\mu\right)1_{A_n}.$$
\end{proof}

\begin{lemma} \label{lem:L1/1R}
Let $\mu$ be a finite measure on a measurable space $Y$, and let $\R1 \subset L^1(\mu)$ denote subspace consisting of constant functions. Then there exists a measurable subset $X \subset Y$ such that $L^1(\mu)/\R1$ is 128-isomorphic to $L^1(\mu \llcorner X)$.
\end{lemma}

\begin{proof}
Let $K$ be as in the statement of Lemma \ref{lem:kernel}. Then by that lemma, it suffices to prove that $L^1(\mu)/\R1$ is 2-isomorphic to $K$. This is achieved by the norm-2 map $L^1(\mu)/\R1 \to K$ given by $g+\R1 \mapsto g-\fint_Y g \, d\mu$ with norm-1 inverse $h \mapsto h+\R1$.
\end{proof}

\section*{Acknowledgements}
We are indebted to the anonymous referee for a very careful reading of the original manuscript and for pointing out numerous mistakes. In particular, we are very thankful for their discovery of a crucial error in the original definition of geometric tree-like decomposition, which caused the original formulation of Theorem 5.7 to be incorrect.

\bibliography{references.bib}
\bibliographystyle{amsalpha}

\end{document}